\documentclass[onefignum,onetabnum]{siamart190516}

\usepackage{amsmath}

\usepackage{amsfonts, psfrag, graphicx, amssymb,enumitem,indentfirst,float,geometry,color,enumerate,graphicx,subcaption,algorithm,algpseudocode,pifont,hyperref,mathtools,mathrsfs}
\graphicspath{{./Figure/}}
\usepackage{mathbbol}
\usepackage{pbox}
\usepackage{booktabs, cellspace, hhline}
\usepackage{makecell}
\allowdisplaybreaks[4]
\numberwithin{equation}{section}
\numberwithin{figure}{section}

\usepackage{etoolbox}
\patchcmd{\thebibliography}{\chapter*}{\section*}{}{}
\usepackage[title]{appendix}

\usepackage{lipsum}

\newsiamremark{assumption}{Assumption}
\newsiamremark{remark}{Remark}

\newcommand{\commentout}[1]{{}} 

\newcommand{\abs}[1]{\left|#1\right|}

\newcommand{\bfa}{{\bf a}}

\newcommand{\bfH}{{\bf H}}

\newcommand{\bfn}{{\bf n}}

\newcommand{\bfP}{{\bf P}}
\newcommand{\bfp}{{\bf p}}

\newcommand{\bft}{{\bf t}}

\newcommand{\bfu}{{\bf u}}

\newcommand{\bfv}{{\bf v}}

\newcommand{\bfx}{{\bf x}}

\newcommand{\bfz}{{\bf z}}

\newcommand{\conv}{\operatorname{conv}}
\newcommand{\ddiv}{\operatorname{div}}
\newcommand{\rot}{\operatorname{rot}}

\newcommand{\jump}[1]{\left[\!\left[#1\right]\!\right]}
\newcommand{\dd}{\,{\rm d}}

\newcommand{\vertiii}[1]{{\left\vert\kern-0.25ex\left\vert\kern-0.25ex\left\vert #1
    \right\vert\kern-0.25ex\right\vert\kern-0.25ex\right\vert}}
    \newcommand{\vertii}[1]{{\left\vert\kern-0.25ex\left\vert #1
    \right\vert\kern-0.25ex\right\vert}}

\begin{document}
\title{
Virtual element methods based on boundary triangulation:\\
fitted and unfitted meshes
}
\author{Ruchi Guo \thanks{The School of Mathematics, Sichuan University, Chengdu, Sichuan, China (ruchiguo@scu.edu.cn).} 
\funding{This work was funded in part by NSF grant DMS-2012465.}
}
\date{}
\maketitle
\begin{abstract}
One remarkable feature of virtual element methods (VEMs) is their great flexibility and robustness
when used on almost arbitrary polytopal meshes. 
This very feature makes it widely used in both fitted and unfitted mesh methods.
Despite extensive numerical studies, 
a rigorous analysis of robust optimal convergence has remained open for highly anisotropic 3D polyhedral meshes. 
In this work, we consider the VEMs in \cite{2022CaoChenGuo,2017ChenWeiWen}
that introduce a boundary triangulation satisfying the maximum angle condition.
We close this theoretical gap regarding  optimal convergence on polyhedral meshes in the lowest-order case 
for the following three types of meshes:
(1) elements only contain non-shrinking inscribed balls but \textit{are not necessarily star convex} to those balls; 
(2) elements are cut arbitrarily from a background Cartesian mesh, which can extremely shrink; 
(3) elements contain different materials on which the virtual spaces involve discontinuous coefficients. 
The first two widely appear in generating fitted meshes for interface and fracture problems,
while the third one is used for unfitted mesh on interface problems.
In addition, the present work also generalizes the maximum angle condition from simplicial meshes to polyhedral meshes.
\end{abstract}




\section{Introduction}

Polyhedral meshes admit many attractive features, especially the flexibility to adapt to complex geometry. 
For existing works on polyhedral meshes, say discontinuous Galerkin methods \cite{2012BassiBottiColombo,2013AntoniettiGianiHouston,2022CangianiDongGeorgoulis,2017CangianiDongGeorgoulisHouston,2014CangianiGeorgoulisHouston,2020PietroDroniou}, 
mimetic finite difference methods \cite{2005BrezziLipnikovShashkov,2005BREZZILIPNIKOVSIMONCINI,2010VeigaManzini}, 
weak Galerkin methods \cite{2012MuWwangYe,2014WangYe} 
and virtual element methods (VEMs) \cite{BEIRAODAVEIGA20171110,Brenner;Sung:2018Virtual,2017VeigaCarloAlessandro,2013BeiraodeVeigaBrezziCangiani,2017AndreaEmmanuilSutton} to be discussed, 
their assumptions for element shape eventually turn into the conventional shape regularity if the polyhedral meshes reduce to simplicial meshes. 
The VEMs were first introduced in  \cite{2013BeiraodeVeigaBrezziCangiani}, 
and the key idea is to develop local problems to construct virtual spaces for approximation. 
VEMs possess several attractive features, especially the conformity to the underlying Hilbert spaces 
and the flexibility to handle almost arbitrary polygonal or polyhedra element shapes. 
Those features bring numerous applications of VEMs in many fields. 
For instances, in \cite{2014BenedettoPieraccini,2022BenvenutiChiozziManziniSukumar,2017ChenWeiWen,2018ThanhZhuangXuan,2018HusseinBlaFadi}, 
VEMs are used on meshes cut by interfaces, fractures and cracks, 
and the convergence is robust to these highly anisotropic meshes, 
which benefits the mesh generation procedure in these problems. 
In fact, it was observed in \cite{BEIRAODAVEIGA20171110} that the optimal convergence of VEMs can be achieved on Voronoi meshes 
of which the control vertices are randomly generated. 
Such flexibility also benefits solving multiscale problems \cite{2021XieWangFeng,2019RivarolaBenedettoLabanda,2020SreekumarTriantafyllouBcot}. 

\begin{table}[h!]
  \centering
  \hspace*{-1cm}
  \begin{tabular}{c | c | c | c| c } 
   \hline
    Poly type       & edge(E)   & face(F)     & body(K) & ref \\
   \hline
   \makecell{\includegraphics[width = 0.5in]{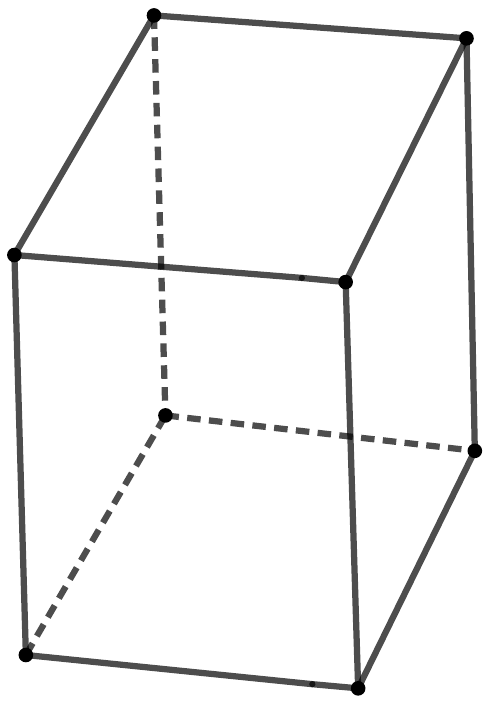} }
  & $\rho_E\simeq h_K$ & $\rho_F\simeq h_K$    & $\rho_K\simeq h_K$  
  & \makecell{ VEM: \cite{2017VeigaCarloAlessandro,2013BeiraodeVeigaBrezziCangiani,2017AndreaEmmanuilSutton} ...\\
  (classical shape regularity) } \\ 
   \hline
   \makecell{\includegraphics[width = 0.5in]{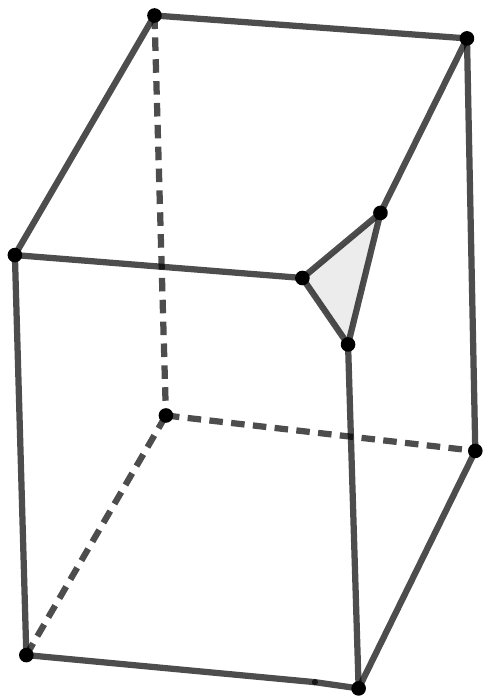}}
  & \makecell{$\rho_E \simeq h_E$\\ ($\rho_E\rightarrow 0$)} & $\rho_F \simeq h_F$ & $\rho_K\simeq h_K$ 
  & \makecell{ \cite{Brenner;Sung:2018Virtual} (small edges and small faces)} \\
   \hline
   \makecell{\includegraphics[width=0.5in]{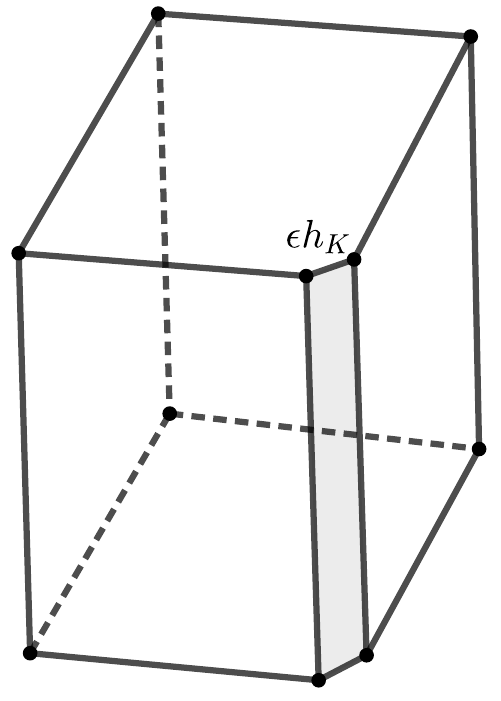}} 
   & $\rho_E\rightarrow 0 $ & $\rho_F\rightarrow 0 $ & $\rho_K\simeq h_K$
   & \makecell{ \cite{Cao;Chen:2018AnisotropicNC} (a height condition)} \\ 
   \hline
  \end{tabular}
  \caption{The existing work for polygonal shapes in VEMs.}
  \label{table_existPolyg}
  \end{table}

However, the error analysis for VEMs on anisotropic meshes (referred to as anisotropic analysis in the following discussion) seems quite challenging. 
Most of the earlier works \cite{2017VeigaCarloAlessandro,2013BeiraodeVeigaBrezziCangiani,2017AndreaEmmanuilSutton} assume such a shape regularity: 
an element $K$ together with each of its faces and edges are all star-convex to non-shrinking balls in 3D, 2D and 1D, respectively; namely
\begin{equation}
  \label{eq_shaperegu_cond}
  \rho_D = \mathcal{O}(h_K),  ~~~~ \text{with} ~ D ~ \text{being an edge $e$, face $F$ or $K$ itself},
\end{equation}
where $\rho_D$ is the radius of the largest ball in $D$. 
In other words, this assumption requires that the elements, faces and edges must have the same side,
and thus rules out short edges, small faces, and shrinking elements.


Some efforts have been made to relax the shape conditions.
As a fundamental advance, the ``no small face" and ``no short edge" assumptions were relaxed by Brenner et al. in \cite{Brenner;Sung:2018Virtual}.
Compared with \eqref{eq_shaperegu_cond}, the condition in \cite{Brenner;Sung:2018Virtual} can be summarized as
\begin{equation}
  \label{eq_shaperegu_cond_2}
\rho_D = \mathcal{O}(h_D) ~~ \text{with} ~ D=e ~ \text{and} ~ F.
\end{equation}
In other words, the edges can be arbitrarily small, while a face $F$ can be small but cannot be thin, say thin rectangles. 
Notice that $\rho_K = \mathcal{O}(h_K)$ in \eqref{eq_shaperegu_cond} is still needed,
meaning that the element body cannot be shrinking in any sense.
What's more, their estimates involve an unfavorable factor in the error bound: $\ln(1 + \max_F \tau_F )$,
with $\tau_F$ being the ratio of the longest edge and shortest edge. 
A remarkable achievement is made in \cite{2018CaoChen}
where the conditions in \eqref{eq_shaperegu_cond} are completely circumvented for the 2D case.
For the 3D case, the authors in \cite{Cao;Chen:2018AnisotropicNC} considered a nonconforming VEM that avoids the star-convexity assumptions of faces and edges, 
and they proposed a ``height condition" to replace the star convexity of $K$, 
i.e., the height of a face $F$ towards its neighbor elements must be $\mathcal{O}(h_F)$. 
A similar condition is also used in \cite{2014WangYe}.  
However, the method in \cite{Cao;Chen:2018AnisotropicNC} relies on a projection $\Pi_{\omega_K}$ defined onto the whole patch $\omega_K$ of an element $K$ 
whose computation is relatively complex and expensive. 
We refer readers to Table \ref{table_existPolyg} for illustration of these conditions,
where the polyhedra are constructed from cutting cubes.
However, these conditions are far from practical, 
as many useful and commonly occurring anisotropic shapes are still not covered—for example, 
those shown in Figures \ref{fig:anisotropic_element} and \ref{fig:cub_example}.
It is also worthwhile to mention that, many of these works can only handle the energy norm estimates, 
as the ``$\ln$"  factor seems very difficult to be removed when estimating the $L^2$ errors.

\begin{figure}[h]
  \centering
    \begin{minipage}{0.45\textwidth}
  \centering
   \includegraphics[width=1.5in]{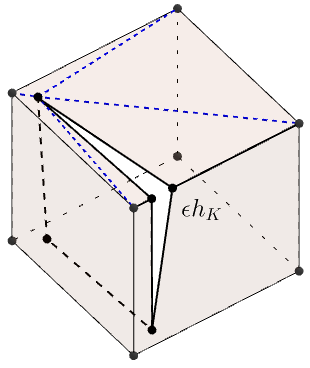}
  \caption{Example of an anisotropic element $K$ which contains a ball of the radius $\mathcal{O}(h_K)$ but is not star convex to it. 
  The face on the left is not supported by an $\mathcal{O}(h_K)$ height towards $K$.
  But its boundary admits a triangulation satisfying Assumptions \hyperref[asp:A1]{A1} and \hyperref[asp:A2]{A2}. 
  This element corresponds to Case (1) studied in this paper.}
  \label{fig:anisotropic_element}
  \end{minipage}
 ~~~ 
  \begin{minipage}{0.45\textwidth}
  \centering
   \includegraphics[width=1.5in]{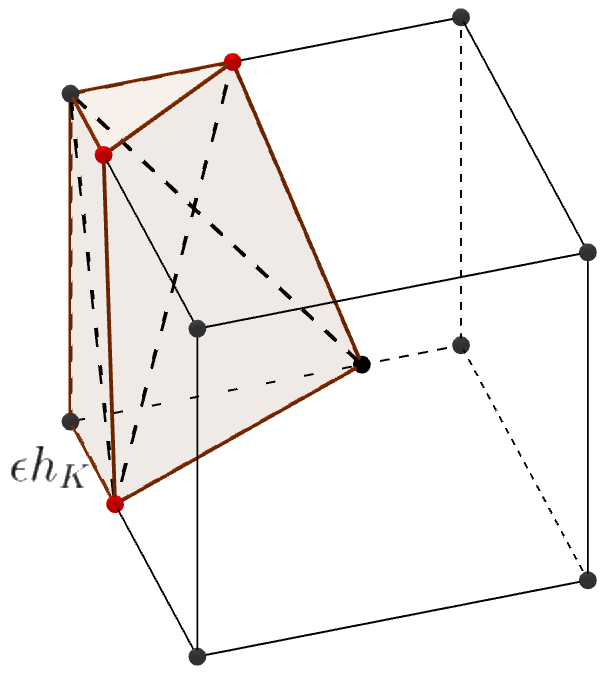}
  \caption{Example of a possibly shrinking element $K$ cut from a cub ($\epsilon \rightarrow 0$). 
  It does not even contain a ball of the radius $\mathcal{O}(h_K)$. 
  In addition, the interior dihedral angle $A_5$-$D_1D_4$-$D_2$ may approach $\pi$ such that the 3D triangulation does not satisfy the MAC, but the boundary triangulation still has the bounded maximum angle.
  This element corresponds to Case (2) studied in this paper. 
  The two subelements of the cube can be also treated together as an entire polyhedron even if they contain discontinuous PDE coefficients,
  i.e., it becomes an interface element of unfitted meshes,
  corresponding to Case (3) studied in this paper.}
  \label{fig:cub_example}
  \end{minipage}
\end{figure}

In summary, the 3D anisotropic analysis of VEMs still remains quite open.
Even how to impose a sufficiently general condition is unknown,
as existing experiments have indicated the robust optimal convergence on meshes beyond the aforementioned shapes.
Meanwhile, for standard finite element methods (FEMs) on simplicial meshes, 
it is well-known that the robust optimal convergence can be achieved under the \textit{maximum angle conditions (MAC)} \cite{1976BabuskaAziz,1999Duran,2020KobayashiTsuchiya,1992Michal,1994Shenk} 
which allows extremely narrow and thin elements.
This observation indicates that there is considerable potential for broadening the class of polyhedral shapes used in VEMs, 
as none of the existing shape conditions can accommodate very thin polyhedra for VEMs or any other methods.

In this work, we argue that the extension of MAC from simplicial meshes to polyhedral meshes is possible for VEMs, 
thereby allowing the shape conditions to be largely relaxed. 
The key idea is to utilize a boundary triangulation of elements satisfying MAC 
for constructing virtual functions—a concept originally introduced in \cite{2022CaoChenGuo,2017ChenWeiWen}.
This is motivated by the practice that constructing such a 2D triangulation on polygonal faces
is much easier than constructing a 3D one.
Indeed, while any 3D triangulation satisfying MAC must result in a 2D boundary triangulation with a bounded maximum angle,
the reverse is not generally true—even for simple prisms, 
see Figure \ref{fig:cub_example} for a counter-example.
In addition, we refer readers to \cite{2021CaoChenGuo} for using a virtual mesh satisfying MAC to show anisotropic analysis of VEMs in the 2D case.

We shall concentrate on the model problem: 
find $u\in H^1_0(\Omega)$ such that
\begin{subequations}
\label{model}
\begin{align}
\label{inter_PDE}
a(u,v) := (\beta \nabla u, \nabla v)_{\Omega} = (f,v)_{\Omega}, ~~~ \forall v\in H^1_0(\Omega),
\end{align}
\end{subequations}
where $\Omega$ is a 3D domain, $f\in L^2(\Omega)$, and $\beta$ can be potentially piecewise constant.
In this work we focus on the lowest-order VEMs,
showing their robust optimal convergence for the three types of meshes:
\begin{itemize} 
 \item[(1)] elements contain \textit{but not necessarily star-convex} to non-shrinking balls; 
 \item[(2)]  elements are cut arbitrarily from a background Cartesian mesh which may not even contain non-shrinking balls; 
 \item[(3)]  elements contain different materials on which the virtual spaces involve discontinuous coefficients.  
\end{itemize}
Figures \ref{fig:anisotropic_element} and \ref{fig:cub_example} illustrate these three cases. 
Note that none of the aforementioned studies in the literature can address any of these cases, 
as they no longer rely on star convexity or height conditions. 
In all these scenarios, short edges and small faces are permitted, and in cases (2) and (3), even very thin elements(subelements) are allowed.
In particular, Case (3) refers to immersed virtual element methods (IVEMs) \cite{2021CaoChenGuoIVEM,2022CaoChenGuo} that involve discontinuous coefficients on unfitted meshes.
The key idea is to project the solutions of local interface problems to the immersed finite element (IFE) spaces which are not standard polynomials \cite{2023ChenGuoZou,2020GuoZhang,2005KafafyLinLinWang} 
such that the jump behavior can be captured.
Related works can be found in \cite{2019BenvenutiChiozziManziniSukumar,2024ZhengChenWang,2022BenvenutiChiozziManziniSukumar,2024DroniouManziniYemm},
but restricted to 2D analysis.
For the 3D case, our key technique is the boundary triangulation with which each interface element is still treated as a polytopal element.
Compared with the standard IFE methods, the advantages of IVEM are unique: the existence of shape functions can be always guaranteed.
In this work, we also compare IVEMs and VEMs for interface problems, examining both their error robustness and the performance of their solvers.

 

The analysis of the three anisotropic cases described above can facilitate many applications of VEMs. 
For instance, case (1) may appear for simulating crack propagation with a background shape regular meshes \cite{2014BenedettoPieraccini,2022BenvenutiChiozziManziniSukumar,2018HusseinBlaFadi}, see Figure \ref{fig:anisotropic_element} for an example.
Case (2) may appear when solving interface problems on a background Cartesian mesh 
with a fitted mesh formulation \cite{2017ChenWeiWen,2009ChenXiaoZhang,1998ChenZou}, 
while Case (3) is for an unfitted mesh formulation \cite{2020AdjeridBabukaGuoLin,2015BurmanClaus,2021CaoChenGuoIVEM,2022CaoChenGuo,2020GuoLin}. 
See Figure \ref{fig:cub_example} for illustration.
The application of VEMs to these problems is particularly advantageous for moving interfaces \cite{2020Guo,2018GuoLinLinElasto,2022MaZhangZheng} and growing cracks \cite{2001DolbowMoesBelytschko,1999MoesDolbowBelytschko}. 


\textbf{The challenge and our contributions.} 
The basic concept of VEMs is to use a discrete local bilinear form 
$a_K(\cdot,\cdot): H^1(K)\times H^1(K) \rightarrow \mathbb{R}$ to approximate \eqref{inter_PDE},
where
\begin{equation}
\label{vem_scheme}
    a_K(u_h, v_h):=(\beta_h \nabla \Pi_K u_h, \nabla \Pi_K v_h)_{K}
    +S_K(u_h-\Pi_K u_h, v_h-\Pi_K v_h).
\end{equation}
Here $\Pi_K$ is a projection of virtual functions onto polynomial spaces,
and $S_K(\cdot,\cdot)$ is a semi-positive symmetric bilinear form, 
called \textit{stabilization}, 
to make $a_K$ stable. 
While the virtual functions are not computable,
the projection and stabilization are computable through DoFs.

Let us recall the meta-framework of analysis in VEMs.
First, to make $a_K$ stable, the stabilization should be strong enough such that the following boundedness holds
for the lowest-order virtual functions $v_h$:
\begin{equation}
\label{stab_bound_old}
S_K(v_h,v_h)  \succsim   | v_h |^2_{H^1(K)}. 
\end{equation}
The precise definition of the stabilization will be given in the next section.
This is typically obtained by the inverse trace inequality and the Poincar\'e inequality in the continuous level.
These two useful tools together with the trace inequality serve as the theoretical foundation for the optimal error estimates of VEMs including the approximation capabilities and stability.
Roughly speaking, 
all the existing work establish these inequalities by considering a mapping from polyhedra to certain reference elements or balls.
If the polyhedra is regular in the sense of \eqref{eq_shaperegu_cond},
the mapping becomes a Lipschitz isomorphism whose derivative is upper bounded.
It makes the constants of those inequalities also upper bounded,
only dependent on the ratio of the radius of the circumscribed and inscribed balls.


However, it is worth noting that establishing condition \eqref{stab_bound_old} and other fundamental inequalities at the continuous level is quite demanding. Doing so typically requires highly regular element shapes, such as those in \eqref{eq_shaperegu_cond} or their equivalent variants. As a result, this approach is not suitable for the anisotropic element shapes encountered in practice, including those considered in the present work—creating a gap between theory and computation.
To fill such a gap, we resort to the VEMs based on boundary triangulation \cite{2022CaoChenGuo,2017ChenWeiWen},
the details of which are provided in the next section.
With the maximal angle $\theta_M$ and the number of triangles $|\mathcal{N}_T|$, 
we argue that \eqref{stab_bound_old} can be replaced by a discrete Poincar\'e-type inequality on element boundary:
\begin{equation}
  \label{SK_equiv_0}
  \| v_h \|^2_{0,\partial K} \le 10|\mathcal{N}_T| h_K \sin\left( \frac{\pi-\theta_M}{2+\epsilon } \right)^{-1}  \| v_h \|_{S_K} 
  \end{equation}
  for a small constant $\epsilon$,
see \eqref{SK_equiv_1} and Lemma \ref{lem_disct_poinc} for details and related discussions.
Clearly, the constant in the bound of \eqref{SK_equiv_0} is very explicit.
We show that this inequality solely leads to the stabilization of the considered VEMs in Lemma \ref{lem_norm}.
As for approximation capabilities through interpolations, 
we give their self-contained proof by taking advantages of the delicate geometric properties,
where the conventional trace inequalities or Poincar\'e inequalities are avoided.
Another highlight is that the classical error estimates based on MAC requiring relatively higher regularity \cite{1999Duran,1994Shenk} are also circumvented. 

Unlike the conventional sophisticated analysis framework,
the ``hard analysis'' in the present work is quite tedious,
due to the involved geometric estimates.
To have a more clear streamline of the analysis, 
we develop an abstract framework
by decomposing the whole analysis into several components:
the \textbf{stability}, the \textbf{approximation capabilities of the boundary space},
and the \textbf{approximation capabilities of the virtual space}.
These three components collectively lead to the desired optimal estimate,
while their proof are self-contained and independent with each other to a certain extend.
Under this framework,
the virtual spaces, local problems and projection spaces remain abstract and can be adapted to the underlying problems' nature. 
For instance, the local problems can involve singular coefficients, 
and the projection spaces are free of polynomials and can be chosen as any computable spaces as long as they can provide sufficient approximation capability,
making the analysis of IVEMs possible..

This article consists of 5 additional sections. 
In the next section, we establish the boundary triangulation and boundary space for constructing the virtual spaces
In Section \ref{sec:assump}, we present an abstract framework for VEMs by introducing several ``hypotheses" and showing these hypotheses can lead to optimal errors. 
In the next 3 sections, we analyze the aforementioned 3 types of VEMs by theoretically verifying the hypotheses. 
In Section \ref{sec:num}, we provide numerical results.
Some technical results will be given in the Appendix.




\section{The VEM based on a boundary triangulation}
\label{sec:unifyframe}

Throughout this article, 
we let $\mathcal{T}_h$ be a polyhedral mesh of $\Omega$,
let $h_K$ be the diameter of an element $K$, and define $h=\max_{K}h_K$. 
We denote the collection of faces and edges of an element $K$ as $\mathcal{F}_K$ and $\mathcal{E}_K$. 
The geometrical conditions of $\mathcal{T}_h$  are left for later discussions in details. 
Let $H^m(D)$ be the standard Sobolev space on a region $D$ and let $H^1_0(D)$ be the space with the zero trace on $\partial D$. 
We further let 
$$
\bfH(\ddiv;D)=\{\bfu\in \bfH(D), \ddiv(u)\in H(D)\} ~~~ \text{and} ~~~ \bfH(\ddiv,c;D)=\{\bfu\in \bfH(D), \ddiv(c u)\in H(D)\},
$$ 
for a piecewise-constant function $c$.
In addition, $\|\cdot\|_{m,D}$ and $|\cdot|_{m,D}$ denote the norms and semi norms. 
The $L^2$ inner product is then denoted as $(\cdot,\cdot)_D$.
For simplicity's sake, we shall employ the notation $\lesssim$ and $\gtrsim$ representing $\le C$ and $\ge C$ where $C$ is a generic constant independent of element shape and size. 
Furthermore, the notation $\simeq$ denotes equivalence where the hidden constant $C$ has the same property. 





\subsection{The boundary triangulation and geometry assumptions}
One distinguished feature of the VEM in \cite{2022CaoChenGuo,2017ChenWeiWen} is a boundary triangulation that enables us to overcome the difficulty arising from anisotropic element shapes. 
Given each triangle $T$, denote $\theta_m(T)$ and $\theta_M(T)$ as the minimum and maximum angles of $T$. 
Recall that a mesh satisfies the maximal and minimal angle condition (MAC and mAC) if, for every $T$ in this mesh, $\theta_M(T)$ and $\theta_m(T)$ are bounded above and below from $\pi$ and $0$, respectively.
With this set-up, we first make the following assumption:
\begin{itemize}
  \item[(\textbf{A1})] \label{asp:A1} For each element $K$, the number of edges and faces is uniformly bounded. 
  Each of its face admits a triangulation in which the edges are connected by vertices only in $\mathcal{N}_K$.
  The collection of these triangles is referred to as the boundary triangulation
  which is assumed to satisfy MAC with the upper bound $\theta_M$.
\end{itemize}  
We shall denote the boundary triangulation by $\mathcal{T}_h(\partial K)$,
and let the collection of all the vertices be $\mathcal{N}_h(\partial K)$.

With Assumption \hyperref[asp:A1]{(A1)}, we define the trace space
\begin{equation}
\label{BhK_fitted}
\mathcal{B}_h(\partial K) = \{ v_h \in L^2(\partial K)~:~ v_h|_T \in \mathcal{P}_1(T), ~ \forall T\in \mathcal{T}_h(\partial K) \},
\end{equation}
in which the functions serve as the boundary data/conditions for the defining the virtual element spaces.
Note that the functions in $\mathcal{B}_h(\partial K)$ are standard finite element functions instead of those defined through local PDEs. 
This is particularly advantageous for establishing the shape-independent stability of the VEM.
In particular, it is the key to achieve the Poincar\'e inequality in \eqref{SK_equiv_0} with a shape-independent constant.

Moreover, the MAC alone may lead to extremely irregular mesh.
We further need the following \textit{path condition}. 
Here, a \textit{path} between two vertices is defined as a collection of edges connecting these two vertices.  
\begin{itemize}
  \item[(\textbf{A2})] \label{asp:A2} Let $T_M\in\mathcal{T}_h(\partial K)$ have the maximum area. 
  For each vertex $\bfz\in \mathcal{N}_K$, there exists a vertex $\bfz'$ of $T_M$ and a path from $\bfz$ to $\bfz'$ such that, for each edge $e$ in this path,  one of its opposite angles $\theta_e$ satisfies $\theta_e \le (1+\epsilon) \theta_m(T)$ where $T$ is the element containing the angle $\theta_e$.
\end{itemize} 

Note that Assumption \hyperref[asp:A2]{(A2)} does not pose any restrictions on the minimum angle which could be still arbitrarily small. 
Roughly speaking, under the assumption, 
two neighborhood elements of an edge in a path cannot both shrink to this edge.
We use Figure \ref{fig:path} for illustration.
Verifying Assumption \hyperref[asp:A2]{(A2)} can sometimes be challenging, 
as it requires a global check over all possible paths—a process that can be computationally expensive when dealing with many triangles.
To address this issue, we introduce an alternative assumption, Assumption \hyperref[asp:A2plus]{(A2')}, provided in Appendix \ref{append:lem_A2plus}. This assumption is much easier to verify since it only requires local information. In most situations, Assumption \hyperref[asp:A2plus]{(A2')} is sufficient, 
for example, those elements that may shrink to a face, as shown in the middle and right plots of Figure \ref{fig:cub}.
But for some extreme case that elements may shrink to an edge, such as the left plot in Figure \ref{fig:cub} as well as Figure \ref{fig:cutmeshA2},
we still need to use Assumption \hyperref[asp:A2]{(A2)}.
The interested readers can easily check that Assumption \hyperref[asp:A2plus]{(A2')} does not hold there. 

\begin{figure}[h]
  \centering
  \begin{minipage}{.4\textwidth}
  \centering
  \includegraphics[width=2.3in]{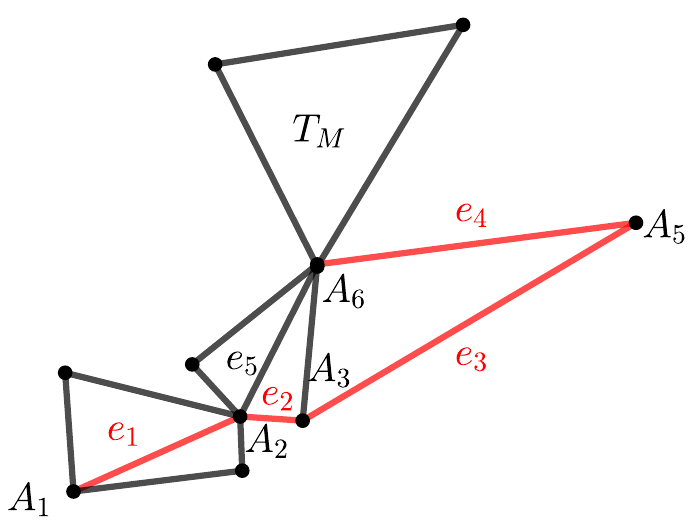}
  \caption{Illustration of Assumption \hyperref[asp:A2]{(A2)}: $e_5 = A_2A_6$ is not allowed in a path as the two neighborhood elements may shrink to this edge. Then, $e_2$, $e_3$ and $e_4$ are needed to connect $A_2$ and $A_6$.}
  \label{fig:path}
  \end{minipage}
  ~~~~
  \begin{minipage}{0.4\textwidth}
   \includegraphics[width=1in]{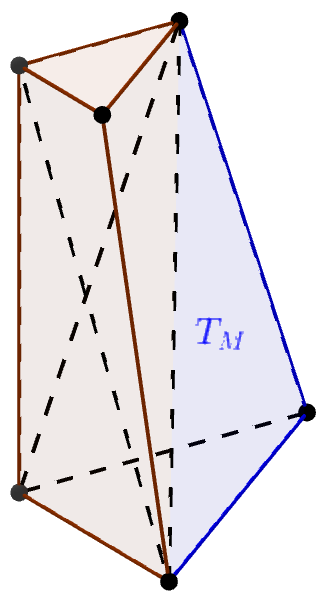}
  \caption{Configuration of an element which may shrink to a vertical edge.
  This element satisfies Assumption \hyperref[asp:A2]{(A2)} but not Assumption \hyperref[asp:A2plus]{(A2')}.}
  \label{fig:cutmeshA2}
  \end{minipage}
\end{figure}


\subsection{The VEM}

In this work, we consider the scenario of $\beta$ being a piecewise-constant function on two subdomains $\Omega^{\pm}$:
\begin{equation}
\label{eq_beta}
\beta = 
\begin{cases}
\beta^- & ~~~ \text{in}~ \Omega^-,\\
\beta^+ & ~~~ \text{in}~ \Omega^+,
\end{cases}
\end{equation}
where the generalization is to more subdomains is straightforward.
We introduce $\beta_h$ as an approximation to $\beta$, 
with errors arising from the geometric approximation of the interface. 
For instance, $\partial\Omega^{\pm}$ maybe a curvilinear interface in inhomogeneous media, 
then $\beta_h$ is a piecewise-constant function but partitioned by a polygonal approximation of the interface surface.
See detailed discussion of this case in Section \ref{sec:IVEM}.

Mimicking the Ciarlet's finite element triplets \cite{Ciarlet1975LecturesOT}, 
for a given an element $K$,
we introduce quintuplets for the description of basic ingredients of a VEM: 
$(K,\mathcal{B}_h(\partial K), \mathcal{V}_h(K), \mathcal{W}_h(K),\mathcal{D}_K)$,
where $\mathcal{B}_h(\partial K)$ is given in \eqref{BhK_fitted}.
Now, let us describe $\mathcal{V}_h(K)$, $\mathcal{W}_h(K)$ and $\mathcal{D}_K$.
With these preparations, we define
\begin{itemize}
\item $\mathcal{V}_h(K)$ is a finite-dimensional virtual element space defined as
\begin{equation}
\label{lifting}
\mathcal{V}_h(K) = \{  \nabla\cdot(\beta  \nabla v_h) =0: v_h|_{\partial K}\in \mathcal{B}_h(\partial K) \} \subset H^1(K);
\end{equation}
\item $\mathcal{W}_h(K) \subset H^1(K)$ is a computable finite-dimensional space (allowed to be a non-polynomial space) containing $\mathcal{P}_0(K)$ for projecting $\mathcal{V}_h(K)$,
satisfying 
\begin{equation}
  \label{lifting_h}
\nabla\cdot(\beta_h \nabla w_h) = 0, ~~~ \forall w_h \in \mathcal{W}_h(K);
\end{equation}
\item $\mathcal{D}_K=\{ v_h(\bfx) \}_{\bfx \in \mathcal{N}_K}$ corresponds to the nodal DoFs. 
\end{itemize}
As $\mathcal{B}_h(\partial K)$ consists of merely standard 2D FE spaces, and the local problem in \eqref{lifting} gives the unique harmonic extension for each boundary function, 
we know that $\mathcal{V}_h(K)$ is unisolvent with respect to the DoFs in $\mathcal{D}_K$.
Then, the global space is defined as 
\begin{equation}
\label{global_space}
\mathcal{V}_h = \{ v_h \in H^1_0(K) :~ v_h|_K \in \mathcal{V}_h(K) , ~ \forall K\in \mathcal{T}_h \} .
\end{equation}

As functions in $\mathcal{V}_h(K)$ are not computable, 
we need a projection $\Pi_K: H^1(K) \rightarrow \mathcal{W}_h(K)$ defined as
\begin{equation}
\label{Pi_proj}
( \beta_h \nabla \Pi_K v_h, \nabla w_h)_{K} = ( \beta_h \nabla v_h, \nabla w_h)_{K} , ~~~~ \forall w_h \in \mathcal{W}_h(K),
\end{equation}
where $\int_{\partial K} \Pi_K v_h \dd s = \int_{\partial K}  v_h \dd s$ is imposed for uniqueness. 
Since $ w_h\in \mathcal{W}_h(K)$ is explicitly known, the projection in \eqref{Pi_proj} is computable:
\begin{equation}
\begin{split}
\label{Pi_proj_compute}
( \beta_h \nabla \Pi_K v_h, \nabla w_h)_{K} & = (  \nabla v_h, \beta_h \nabla w_h)_{K} \\
& = -  (   v_h, \text{div}( \beta_h \nabla w_h) )_{K} + (\beta_h \nabla w_h\cdot\bfn, v_h)_{0,\partial K} =  (\beta_h \nabla w_h\cdot\bfn, v_h)_{0,\partial K},
\end{split}
\end{equation}
provided known $v_h|_{\partial K}$. 
Then, the global bilinear form is defined as
\begin{equation}
\label{vem_scheme_glob}
a_h(u_h, v_h):=\sum_{K\in \mathcal{T}_h} a_K(u_h, v_h),
\end{equation} 
which leads to a reasonably good approximation to $a(u_h, v_h)$, $\forall u_h, v_h \in H^1_0(\Omega)$.
Now, the virtual scheme is to find $u_h\in \mathcal{V}_h$ such that 
\begin{equation}
\label{eq_VEM}
    a_h(u_h, v_h) =\sum_{K\in \mathcal{T}_h}(f, \Pi_K v_h)_{L^2(K)}, ~~~~~ \forall v_h \in \mathcal{V}_h.
\end{equation}





\begin{remark}
\label{rem_discont}
For standard VEMs, the typical choice of $\mathcal{W}_h(K)$ is $\mathcal{P}_1(K)$. 
For IVEMs, $\mathcal{W}_h(K)$ becomes the IFE space.
\end{remark}



\subsection{The stabilization and a discrete Poincar\'e type inequality}
\label{subsec:stab}
We recall the following projection estimate which will be frequently used in this work.
Given a domain $D$ and a non-negative integers $m$, let be $\mathrm{P}^m_D$ the $L^2$ projection form $L^2(D)$ to $\mathcal{P}_m(D)$. 
\begin{lemma}[\cite{1999Rudiger}]
\label{lem_proj}
Let $D$ be a domain and let $k = 0,1,..., m$. Then, for every $u\in H^{m+1}(D)$, there is
\begin{subequations}
\label{lem_proj_eq0}
\begin{align}
    &   \| u -  \mathrm{P}^m_D u \|_{k,D} \leqslant c_{m,k} h^{m+1-k}_D |u|_{m+1,D},  ~~~~~~~~ \text{for convex}~D, \label{lem_proj_eq01}  \\
    &  \| u -  \mathrm{P}^m_D u \|_{0,D} \leqslant c_{m,0} h^{m+1}_D |u|_{m+1,\conv(D)},  ~~~~~ \text{for arbitrary}~D, \label{lem_proj_eq02}  
\end{align}
\end{subequations}
where the constant $c_{m,k}$ is independent of the geometry of $D$.
\end{lemma}
\begin{proof}
The convex case is given in \cite{1999Rudiger} where the constant $c_{m,k}$ is specified by (1.1) in the reference.
For the non-convex case, the projection property 
$$
\| u -  \mathrm{P}^m_D u \|_{0,D} \le \| u -  \mathrm{P}^m_{\conv(D)} u \|_{0,D} \le \| u -  \mathrm{P}^m_{\conv(D)} u \|_{0,\conv(D)},
$$
finishes the proof.
\end{proof}

With the boundary space in \eqref{BhK_fitted}, in this work we consider the stabilization
\begin{equation}
\label{stab_f}
S_K(v_h,w_h) = h_K \sum_{F\in\mathcal{F}_K} (\nabla_F v_h,\nabla_F w_h)_{0,F}
\end{equation}
which is computable since functions in $\mathcal{B}_h(\partial K)$ are known. 


The trace inequality with \eqref{stab_bound_old} links $\| v_h \|_{0,\partial K} $ and $\| v_h \|_{S_K}$, leading to the stability of many VEMs.
In the following lemma, we proceed to show the stability by directly connecting $S_K$ and $\| \cdot \|_{0,\partial K}$, i.e., \eqref{SK_equiv_0}, 
without relying on the trace inequality and any star-convexity conditions.
The key to show this is to utilize the discrete properties of the space $ \mathcal{B}_h(\partial K)$
which the continuous Sobolev spaces do not have.

\begin{lemma}
\label{lem_oppo_angle}
Suppose a triangle $T$ has the maximum angle $\theta_M(T)$. 
Let $e$ be one of its edges with the opposite angle $\theta_e$ and the ending points $\bfx_1$ and $\bfx_2$. 
Suppose $\theta_e\le (1+\epsilon) \theta_m(T)$. Then, there holds
\begin{equation}
\label{lem_oppo_angle_eq0}
| w_h(\bfx_2) - w_h(\bfx_1) | \le \kappa \| \nabla w_h \|_{0,T} ~~~~ \forall w_h \in \mathcal{P}_1(T),
\end{equation}
\end{lemma}
where $\kappa = \sqrt{2/\sin\left( \frac{\pi-\theta_M(T)}{2+\epsilon } \right)}$.
\begin{proof}
If $\theta_e$ itself is the minimum angle, then, both the edges opposite to $\bfx_1$ and $\bfx_2$ should have the length greater than $|e|$.
Thus, we have $|T|\ge \sin(\theta_M(T))|e|^2/2$, and then obtain
\begin{equation}
\label{lem_oppo_angle_eq1}
| w_h(\bfx_2) - w_h(\bfx_1) | = \abs{ \int_{\bfx_1}^{\bfx_2} \partial_{\bft_e} w_h \dd s } \le |e| \| \nabla w_h \| \le \sqrt{2/\sin(\theta_M)} \| \nabla w_h \|_{0,T},
\end{equation}
which yields \eqref{lem_oppo_angle_eq0}, as $ \frac{\pi-\theta_M}{2+\epsilon } \le \min\{ \theta_M, \pi-\theta_M \}$.

If $\theta_e$ is not the minimum angle, by the assumption we have $(2+\epsilon)\theta_m(T) + \theta_M(T) \ge \pi$ implying $\theta_m(T) \ge (\pi-\theta_M(T))/(2+\epsilon)$. 
Then, by the sine law, there holds 
$
|T| = |e|^2 \sin(\theta_m)\sin( \theta_e + \theta_m )/(2\sin(\theta_e)),
$ 
and notice
$$
\frac{\sin( \theta_e + \theta_m )}{\sin(\theta_e)} = \sin(\theta_m)\cot(\theta_e) + \cos(\theta_m) \le 2 \cos(\theta_m) \le 2,
$$
where we have used $\theta_e\ge \theta_m$. It yields $|T| \le |e|^2 \sin(\theta_m)$.
Hence, we obtain
\begin{equation}
\label{lem_oppo_angle_eq2}
| w_h(\bfx_2) - w_h(\bfx_1) | \le |e| \| \nabla w_h \| \le \sqrt{ 1/ \sin(\theta_m) } \| \nabla w_h \|_{0,T}.
\end{equation}
Then, the desired estimate follows from \eqref{lem_oppo_angle_eq2} with the bound of $\theta_m$.
\end{proof}

\begin{lemma}
\label{lem_disct_poinc}
Given a polyhedral element $K$, 
suppose $\partial K$ has a boundary triangulation $\mathcal{T}_h(\partial K)$ satisfying the Assumptions \hyperref[asp:A1]{(A1)} and \hyperref[asp:A2]{(A2)}. Then, there holds
\begin{equation}
\label{lem_disct_poinc_eq0}
\| v_h - \mathrm{P}^0_{\partial K} v_h \|_{0,\partial P} \le \sqrt{5} \kappa h_K |N_{\mathcal{T}}|^{1/2}  |v_h|_{1,\partial K}, ~~~~ \forall v_h\in \mathcal{B}_h(\partial K),
\end{equation}
where $\kappa$ inherits from Lemma \ref{lem_oppo_angle} with $\theta_M(T)$ replaced by the maximum angle $\theta_M$ in Assumption \hyperref[asp:A1]{(A1)}, and $N_{\mathcal{T}}$ is the number of elements in the boundary triangulation.
\end{lemma}
\begin{proof}
To simplify the notation, $\nabla$ in this proof is understood as the surface gradient on $\partial K$. 
Let $T_M$ be the triangle with the maximum area, and we trivially have
\begin{equation}
\label{lem_disct_poinc_eq1}
\| v_h - \mathrm{P}^0_{\partial K} v_h \|_{0,\partial K} \le \| v_h - \mathrm{P}^0_{T_M} v_h \|_{0,\partial K}.
\end{equation}
Let $\tilde{v}_h = v_h - \mathrm{P}^0_{T_M} v_h$, and 
let $\bfz$ be the vertex at which $\tilde{v}_h$ achieves the maximum value on $\partial K$. 
Consider the path from Assumption \hyperref[asp:A2]{(A2)} connecting $\bfz$ and one vertex $\bfz'$ of $T_M$. 
Let the path be formed by $\{e_l\}_{l=1}^L$ with the neighborhood triangles $\{T_l\}_{l=1}^L$ described by Assumption \hyperref[asp:A2]{(A2)}.
Then, Lemma \ref{lem_oppo_angle} implies 
$$
|\tilde{v}_h(\bfz)| \le | \tilde{v}_h(\bfz) - \tilde{v}_h(\bfz') |  + |\tilde{v}_h(\bfz')| 
\le  \kappa \sum_{l=1}^L \| \nabla \tilde{v}_h \|_{0, T_l} +  |\tilde{v}_h(\bfz')|,
$$
which, by AM–GM inequality, further yields
\begin{equation}
\begin{split}
\label{lem_disct_poinc_eq2}
\| \tilde{v}_h \|^2_{0,\partial K} \le |\partial K| |\tilde{v}_h(\bfz)|^{2} \le 2 |\partial K| |\tilde{v}_h(\bfz')|^{2} + 2 L  |\partial K| \kappa^2 \| \nabla \tilde{v}_h \|^2_{0,\cup_{l=1}^L T_l}.
\end{split}
\end{equation}
As $\tilde{v}_h$ must vanish at one point in $T_M$, and it is a linear polynomial, 
we simply have $\| \tilde{v}_h \|_{\infty,\partial T_M} \le h_{T_M}|T_M|^{-1/2} \| \nabla v_h \|_{0,T_M}$. 
Noticing $L \le 3| {N}_\mathcal{T} |/2$, we derive from \eqref{lem_disct_poinc_eq2} that
\begin{equation}
\label{lem_disct_poinc_eq3}
\| \tilde{v}_h \|^2_{0,\partial K} \le 2|\partial K|/|T_M| h^2_K \| \nabla v_h \|^2_{0,T_M} + 3| {N}_\mathcal{T} |  |\partial K| \kappa^2 \| \nabla v_h \|^2_{0,\partial K}
\le 5| {N}_\mathcal{T} | h^2_K \kappa^2 \| \nabla v_h \|_{0,\partial K},
\end{equation}
where we have also used $|\partial K|/|T_M| \le | {N}_\mathcal{T} |$. It finishes the proof.
\end{proof}

Then, we conclude the most important result in this subsection.
\begin{lemma}
\label{lem_stab_f_verify}
Under Assumptions \hyperref[asp:A1]{(A1)} and \hyperref[asp:A2]{(A2)}, \eqref{SK_equiv_0} holds $\forall v_h \in \mathcal{B}^0_h(\partial K)$.
\end{lemma}
\begin{proof}
The result immediately follows from Lemma \ref{lem_disct_poinc}.
\end{proof}

\begin{remark}
\label{rem_disct_poinc}
For a 2D polygonal region $D$ with a triangulation $\mathcal{T}_h(D)$, we can show a similar estimate. 
Remarkably, the constants in these estimates are very explicitly specified and independent of the shape. 
It is worth mentioning that the constant in \eqref{lem_disct_poinc_eq0} tends to $\infty$ if the triangulation is very fine. 
This property aligns with the behavior of the classical Poincar\'e inequality for irregular domains, 
in the sense that the discrete space approaches the $H^1$ space as the mesh is refined.
\end{remark}

\section{General analysis framework}
\label{sec:assump}
The purpose of this section is to establish a concise and streamlined argument demonstrating that the suitable approximation properties of $\mathcal{V}_h$ and $\mathcal{W}_h$ 
(described by Hypotheses \hyperref[asp:H3]{(H2)}-\hyperref[asp:H5]{(H4)}) with a discrete Poincaré inequality (described by Hypothesis \hyperref[asp:H2]{(H1)}) are sufficient to derive the optimal error estimate.
It is important to emphasize that no specific geometric assumptions or trace/inverse inequalities are required in this section.
The detailed proofs of these hypotheses will be provided in Sections \ref{sec:fitted_1}-\ref{sec:IVEM} for each considered anisotropic meshes. 
There, readers will find how delicate geometric properties (rather than shape regularity) are employed in each case to establish these hypotheses.

\subsection{The hypotheses and corollaries}
We begin with defining the following quantities from $a_K(\cdot,\cdot)$:
\begin{equation*}
\label{energy_norm}
\vertiii{v_h}^2_K = a_K(v_h,v_h), ~~~ \text{and} ~~~ \vertiii{v_h}^2_h = a_h(v_h,v_h).
\end{equation*}
For some suitable $u$ with sufficient regularity admitting pointwise evaluation,
we define the interpolation $I_Ku\in \mathcal{V}_h(K)$ by $I_K u(\bfx) = u(\bfx)$, $\forall \bfx \in \mathcal{N}_h(\partial K)$.
Then, $u_I$ denotes the global interpolation.

\begin{itemize}
  \item[\textbf{(H1)}]\label{asp:H2} (Stability) The stabilization $S_K(\cdot,\cdot)$ is non-negative and leads to a norm $\|\cdot\|_{S_K}$ 
on $\widetilde{\mathcal{B}}^0_h(\partial K): = \{v_h \in \mathcal{B}_h(\partial K) \oplus \text{Tr}_{\partial K} ~ \mathcal{W}_h(K) : (v_h,1)_{\partial K} = 0\}$ such that
\begin{equation}
\label{SK_equiv_1}
\| \cdot \|_{0,\partial K} \lesssim h^{1/2}_K\| \cdot \|_{S_K} ~~~ \text{in} ~ \widetilde{\mathcal{B}}^0_h(\partial K),
\end{equation}
where $\text{Tr}_{\partial K} $ denotes the trace operator on $\partial K$.
  \item[\textbf{(H2)}]\label{asp:H3} (The approximation capabilities of $\mathcal{V}_h(K)$) 
There holds
\begin{equation}
\label{approxi_VK_Pi}
\sum_{K\in \mathcal{T}_h} \vertiii{ u - u_I }^2_K \lesssim h^2 \| u \|^2_{2,\Omega}.
\end{equation}
\item[\textbf{(H3)}] \label{asp:H4} (The approximation capabilities of $\mathcal{W}_h(\partial K)$) There hold
\begin{subequations}
\label{Pi_approx}
\begin{align}
    &   \sum_{K\in \mathcal{T}_h}  \| \nabla (u - \Pi_K u) \|^2_{0,K} \lesssim h^2 \|u\|^2_{2,\Omega} ,   \label{Pi_approx_eq1} \\
    &   \sum_{K\in \mathcal{T}_h}  h_K\| \nabla(u - \Pi_K u)\cdot\bfn \|^2_{0,\partial K}  + \| u - \Pi_K u \|^2_{S_K}  \lesssim h^2 \|u\|_{2,\Omega} .   \label{Pi_approx_eq2} 
\end{align}
\end{subequations}
\item[\textbf{(H4)}] \label{asp:H5} (The extra hypothesis for $L^2$ estimates) There holds
\begin{equation}
\label{L2_assump}
  \sum_{K\in \mathcal{T}_h}  \| u - \Pi_K u \|^2_{0,K} + h_K \| u -  u_I \|^2_{0,\partial K} \lesssim h^4 \|u\|^2_{2,\Omega} .
\end{equation}

\item[\textbf{(H5)}] \label{asp:H6} (The approximation of $\beta_h$) There holds $\| \beta_h \|_{\infty,\Omega} \le \| \beta \|_{\infty,\Omega}$, and
\begin{equation}
\label{betah_approx}
 \sum_{K\in\mathcal{T}_h}   \|  |\beta - \beta_h|^{1/2} \nabla u  \|^2_{0,K} + h_K \| |\beta - \beta_h| \nabla u\cdot\bfn   \|^2_{0,\partial K} \lesssim h^2 \|u\|^2_{2,\Omega} .
\end{equation}

\end{itemize}

\begin{remark}
We provide some explanation of Hypotheses above. 
\begin{itemize}


\item 
We will see that \hyperref[asp:H2]{(H1)} merely assists in making $\vertiii{\cdot}_h$ a norm, given by Lemma \ref{lem_norm}, 
and showing the boundedness in Lemma \ref{cor_uI} and \eqref{lem_err_eqn_6}.
For classical VEMs on isotropic meshes, 
usually one also needs \hyperref[asp:H2]{(H1)} to estimate $\|u-u_I\|_{j,K}$ and $\|u- \Pi_K u \|_{j,K}$.
We will see in the next three sections that the analysis of \hyperref[asp:H3]{(H2)}-\hyperref[asp:H5]{(H4)} are independent of \hyperref[asp:H2]{(H1)} or any other norm equivalence results.

\item Note that $\mathcal{W}_h(K)\subset \mathcal{V}_h(K)$ is not required in the present framework. 
But, in most situations, this is indeed true, 
which implies $\text{Tr}_{\partial K} ~ \mathcal{W}_h(K) \subset \mathcal{B}_h(\partial K)$. 
In this case, \hyperref[asp:H2]{(H1)} is just needed for $\mathcal{B}^0_h(\partial K): = \{v_h \in \mathcal{B}_h(\partial K) : (v_h,1)_{\partial K} = 0\}$, 
reduced to Lemma \ref{lem_disct_poinc}.

\item We now proceed to show these Hypotheses can collectively lead to optimal convergence.
Special attention must be paid to that no shape regularity assumption is used in this ``general discussion''.  
In particular, the estimate for the $L^2$ norm in the literature \cite{2013BeiraodeVeigaBrezziCangiani,Brenner;Sung:2018Virtual} heavily relies on the estimate for $\|u-u_I\|_{0,K}$ 
which is not available in the present work due to the anisotropic meshes. 


\end{itemize}
\end{remark}

%

\begin{lemma}
\label{lem_norm}
Under Hypotheses \hyperref[asp:H2]{(H1)}, $\vertiii{\cdot}_h$ defines a norm on $\mathcal{V}_h$.
\end{lemma}
\begin{proof}
Given any $v_h\in\mathcal{V}_h$, assume $\vertiii{v_h}_h = 0$. Then, $\| \nabla \Pi_K v_h \|_{0,K}=0$ implies $\Pi_K v_h\in \mathcal{P}_0(K)$. As $(v_h - \Pi_K v_h)|_{\partial K} \in \mathcal{B}^0_h(\partial K)$, by \hyperref[asp:H2]{(H1)}, we have $v_h|_{\partial K} = \Pi_K v_h|_{\partial K}\in\mathcal{P}_0(\partial K)$. 
Then, we have $v_h\in\mathcal{P}_0(K)$ due to the definition in \eqref{lifting}.
\end{proof}




In the forthcoming discussion, we consider the regularity assumption $u\in H^2_{0}(\beta;\Omega)$ where
\begin{equation}
  \label{beta_space_1}
H^2_{0}(\beta;\Omega) := \{ u\in H^1_0(\Omega)\cap H^2\left(\Omega^+\cup\Omega^-\right): \beta \nabla u \in \bfH(\ddiv;\Omega) \}. 
\end{equation}
The following lemmas will be frequently used in this section.
\begin{lemma}
\label{lem_est_partialK}
Let $u\in H^2_0(\beta;\Omega)$. Under Hypotheses \hyperref[asp:H4]{(H3)} and \hyperref[asp:H6]{(H5)}, there holds
\begin{equation}
\label{lem_est_partialK_eq0}
\sum_{K\in\mathcal{T}_h} (\beta \nabla u\cdot \bfn, \Pi_K v)_{\partial K} \lesssim h^{1/2} \| u \|_{2,\Omega} \left( \sum_{K\in\mathcal{T}_h} \| v -\Pi_K v \|^2_{0,\partial K} \right)^{1/2} , ~~~~ \forall v\in H^1(\Omega).
\end{equation}
\end{lemma}
\begin{proof}
By the property of $\mathcal{W}_h(K)$, we can write down the following identity
\begin{equation*}
\label{lem_est_partialK_eq1}
(\beta_h \nabla \Pi_K u\cdot\bfn, \Pi_K v - v)_{\partial K} = (\beta_h \nabla \Pi_K u , \nabla( \Pi_K v - v)  )_{K}  = 0.
\end{equation*}
Then, we obtain
\begin{equation*}
\begin{split}
\label{lem_est_partialK_eq2}
&(\beta \nabla u\cdot \mathbf{ n}, \Pi_K v - v)_{\partial K} = (\beta_h( \nabla u -  \nabla \Pi_K u)\cdot \mathbf{ n} , \Pi_K v - v)_{\partial K}  + ( (\beta-\beta_h) \nabla u\cdot \mathbf{ n}  , \Pi_K v - v)_{\partial K} \\
 \lesssim & \left( \| \beta_h ( \nabla u -  \nabla \Pi_K u)\cdot \mathbf{ n} \|_{0,\partial K} + \| (\beta-\beta_h) \nabla u\cdot \mathbf{ n} \|_{0,\partial K } \right) \| \Pi_K v_h - v_h \|_{0,\partial K}.
\end{split}
\end{equation*}
The estimate of $\| \beta_h ( \nabla u -  \nabla \Pi_K u)\cdot \mathbf{ n} \|_{0,\partial K}$ is given by \eqref{Pi_approx_eq2} in Hypothesis \hyperref[asp:H4]{(H3)}, while the estimate of $ \| (\beta-\beta_h) \nabla u\cdot \mathbf{ n} \|_{0,\partial K }$ follows from Hypothesis \hyperref[asp:H6]{(H5)}.
\end{proof}

\begin{lemma}
\label{cor_uI}
Let $u\in H^2_0(\beta;\Omega)$. Under Hypotheses \hyperref[asp:H2]{(H1)}-\hyperref[asp:H4]{(H3)} and \hyperref[asp:H6]{(H5)}, there holds
\begin{equation}
\label{cor_uI_eq0}
\sum_{K\in\mathcal{T}_h}  \| u_I - \Pi_K u_I \|^2_{0,\partial K} \lesssim \sum_{K\in\mathcal{T}_h} h_K \| u_I - \Pi_K u_I \|^2_{S_K} \lesssim h^3 \| u \|^2_{2,\Omega} .
\end{equation}
\end{lemma}
\begin{proof}
The first inequality in \eqref{cor_uI_eq0} directly follows from \eqref{SK_equiv_1}. For the second one, notice that 
\begin{equation}
\begin{split}
\label{cor_uI_eq1}
\| u_I - \Pi_K u_I\|_{S_K} & \le \| (u - u_I) - \Pi_K (u - u_I) \|_{S_K} + \| u - \Pi_K u \|_{S_K} 
\end{split}
\end{equation}
of which the estimates follow from Hypothesis \hyperref[asp:H3]{(H2)} and \eqref{Pi_approx_eq2}.
\end{proof}


\subsection{The energy and $L^2$-norm error estimates}
We consider the error decomposition:
\begin{equation}
\label{err_decomp}
\xi_h = u - u_I, ~~~~ \text{and} ~~~~ \eta_h = u_I - u_h,
\end{equation}
where $u_h$ is the VEM solution. 
We first address the energy norm.

\begin{theorem}
\label{lem_err_eqn}
Let $u\in H^2_0(\beta;\Omega)$. Under Hypotheses \hyperref[asp:H2]{(H1)}-\hyperref[asp:H4]{(H3)} and \hyperref[asp:H6]{(H5)}, there holds
\begin{equation}
\label{lem_err_eqn_eq0}
\begin{aligned}
\vertiii{ \eta_h}_h \lesssim h \| u \|_{2,\Omega}.
\end{aligned}
\end{equation}
\end{theorem}
\begin{proof}
Applying integration by parts to the equation $-\nabla \cdot (\beta \nabla u) = f$ tested with $\Pi_K v_h$, $\forall v_h \in \mathcal{V}_h(\Omega)$, we obtain
\begin{equation}
\label{lem_err_eqn_1}
\begin{aligned}
 & a_h(\eta_h,v_h) = a_h(u_h-u_I, v_h) = a_h(u_h, v_h)-a_h(u_I,v_h) \\
=& \sum_{K\in \mathcal{T}_h}[ \underbrace{(\beta \nabla u, \nabla \Pi_K v_h)_K - (\beta_h \nabla \Pi_K u_I, \Pi_K u_h) }_{(I)} - \underbrace{ (\beta \nabla u\cdot \mathbf{ n}, \Pi_K v_h)_{\partial K} }_{(II)}] \\
& -\underbrace{ S_K(u_I-\Pi_K u_I, v_h-\Pi_K v_h) }_{(III)}.
\end{aligned}
\end{equation}
For $(I)$ in \eqref{lem_err_eqn_1}, we have
\begin{equation*}
\begin{split}
\label{lem_err_eqn_2}
(I) 
=  (\beta_h (\nabla u - \nabla  \Pi_K u), \nabla \Pi_K v_h)_K  + ((\beta-\beta_h) \nabla u, \nabla \Pi_K v_h)_K + (\beta_h \nabla \Pi_K (u-u_I), \nabla \Pi_K v_h)_K,
\end{split}
\end{equation*}
of which the estimates follow from \eqref{Pi_approx_eq1} in Hypothesis \hyperref[asp:H4]{(H3)}, Hypotheses \hyperref[asp:H6]{(H5)} and \hyperref[asp:H3]{(H2)}, respectively.
For $(II)$, by Lemma \ref{lem_est_partialK}, we only need to estimate $\| v_h-\Pi_K v_h \|_{0, \partial K}$ which follows from Hypotheses \hyperref[asp:H2]{(H1)}:
\begin{equation}
\begin{split}
\label{lem_err_eqn_6}
\| v_h-\Pi_K v_h \|_{0,\partial K} \lesssim h^{1/2}_K \| v_h-\Pi_K v_h \|_{S_K} \lesssim h^{1/2}_K  \vertiii{v_h}_h.
\end{split}
\end{equation}
As for $(III)$, we note that $(III) \le \| u_I - \Pi_K u_I \|_{S_K} \| v_h - \Pi_K v_h \|_{S_K}$,
and the estimate of $\| u_I - \Pi_K u_I \|_{S_K}$ follows from Lemma \ref{cor_uI}.
Putting all these estimates together into \eqref{lem_err_eqn_1}, and taking $v_h = \eta_h$, we obtain the desired estimate.
\end{proof}

Now, we present the following main theorem.
\begin{theorem}[The energy norm]
\label{thm_energy_est}
Let $u\in H^2_0(\beta;\Omega)$. Under Hypotheses \hyperref[asp:H2]{(H1)}-\hyperref[asp:H4]{(H3)} and \hyperref[asp:H6]{(H5)}, there holds
\begin{equation}
\label{thm_energy_est_eq0}
\vertiii{u-u_h}_h \lesssim h \| u \|_{2,\Omega}.
\end{equation}
\end{theorem}
\begin{proof}
The estimates of $\vertiii{\xi_h}_h$ and $\vertiii{\eta_h}_h$ follow from Hypothesis \hyperref[asp:H4]{(H3)} and Theorem \ref{lem_err_eqn}, respectively.
\end{proof}


The $L^2$ norm estimation under anisotropic elements is more difficult. 
We begin with following two corollaries from the energy norm estimate.

\begin{corollary}
\label{cor_energy}
Let $u\in H^2_0(\beta;\Omega)$. Under Hypotheses \hyperref[asp:H2]{(H1)}-\hyperref[asp:H4]{(H3)} and \hyperref[asp:H6]{(H5)}, there holds
\begin{subequations}
\begin{align}
& \sum_{K\in\mathcal{T}_h} \| \nabla( u - \Pi_K u_h ) \|^2_{0,K} \lesssim h^2 \| u \|^2_{2,\Omega},  \label{cor_energy_eq1}\\
 &  \sum_{K\in\mathcal{T}_h}\| u_h - \Pi_K u_h \|^2_{S_K} \lesssim h^2 \| u \|^2_{2,\Omega}.  \label{cor_energy_eq0}
\end{align}
\end{subequations}
\end{corollary}
\begin{proof}
\eqref{cor_energy_eq1} follows from inserting $u$ into $\| \nabla\Pi_K(u-u_h)\|_{0,K}$. The estimate for \eqref{cor_energy_eq0} follows from a similar decomposition in \eqref{cor_uI_eq1}. 
\end{proof}

We are ready to estimate the solution errors under the $L^2$ norm.
\begin{theorem}
\label{thm_u}
Let $u\in H^2_0(\beta;\Omega)$. Under Hypotheses \hyperref[asp:H2]{(H1)}-\hyperref[asp:H6]{(H5)}, there holds
\begin{equation}
\label{thm_u_eq0}
\| u - \Pi u_h \|_{0,\Omega} \lesssim h^2 \| u \|_{2,\Omega}.
\end{equation}
\end{theorem}
\begin{proof}
Let $z\in H^2_0(\beta;\Omega)$ be the solution to $-\nabla\cdot(\beta\nabla z) = \Pi(u-u_h)$. Testing this equation by $\Pi(u-u_h)$ and applying integration by parts, we have
\begin{equation}
\label{thm_u_eq1}
\| \Pi(u-u_h) \|^2_{0,\Omega} = \sum_{K\in\mathcal{T}_h} \underbrace{ ( \beta \nabla z, \nabla\Pi_K(u-u_h) )_K }_{(I)} - \underbrace{ ( \beta \nabla z\cdot\bfn,  \Pi_K(u-u_h)  )_{\partial K} }_{(II)}.
\end{equation}
For $(I)$, we notice that $( \beta_h \nabla \Pi_K z , \nabla (u-\Pi_K u_h) )_{ K} = ( \beta_h \nabla z , \nabla \Pi_K(u- u_h) )_{ K}$ by the projection property, which yields
\begin{equation}
\begin{split}
\label{thm_u_eq2}
(I) 
&= \underbrace{ ( (\beta- \beta_h) \nabla z, \nabla\Pi_K(u-u_h) )_K }_{(Ia)} + \underbrace{ ( \beta_h \nabla ( \Pi_K z - z ), \nabla (u-\Pi_K u_h) )_{ K} }_{(Ib)} \\
& + \underbrace{ ( \beta_h \nabla   z, \nabla (u-\Pi_K  u_h) )_{ K} }_{(Ic)}.
\end{split}
\end{equation}
Here, $(Ia)$ follows from Hypothesis \hyperref[asp:H6]{(H5)} and Theorem \ref{thm_energy_est}, 
while $(Ib)$ follows from \eqref{cor_energy_eq1} in Corollary \eqref{cor_energy} and \eqref{Pi_approx_eq1} in Hypothesis \hyperref[asp:H4]{(H3)} with $\|\beta_h\|_{\infty}\le \|\beta \|_{\infty}$ in Hypothesis \hyperref[asp:H6]{(H6)}.

The most difficult one is $(Ic)$, for which we use the projection property to write down
\begin{equation}
\label{thm_u_eq5}
(Ic) = \underbrace{ ( (\beta_h - \beta) \nabla z, \nabla u )_{ K} }_{\dagger_0} + \underbrace{ ( \beta \nabla z, \nabla u )_{ K} }_{\dagger_1} - \underbrace{ ( \beta_h \nabla \Pi_K z, \nabla \Pi_K u_h )_{ K}  }_{\dagger_2}.
\end{equation}
For $\dagger_0$, we need to apply Hypothesis \hyperref[asp:H6]{(H5)} to both $z$ and $u$:
$$
\sum_K \dagger_0 \le \left( \sum_K \| |\beta_h - \beta|^{1/2} \nabla z \|^2_K \right)^{1/2} \left( \sum_K \| |\beta_h - \beta|^{1/2} \nabla u \|^2_K \right)^{1/2} \le h^4 \| u \|^2_{2,\Omega} \| z \|^2_{2,\Omega}.
$$ 
The estimation of the remaining two terms are much more involved.
We first notice the following identity by inserting $ \nabla \Pi_Kz_I$ and $ \nabla \Pi_K u_h$:
\begin{equation*}
\begin{split}
\label{thm_u_eq6_1}
\dagger_2 &= ( \beta_h \nabla \Pi_K u_h, \nabla \Pi_K z_I )_{ K} +  ( \beta_h (\nabla \Pi_K u_h - \nabla u)), \nabla \Pi_K(z- z_I) )_{ K} \\
& + ( \beta \nabla  u, \nabla \Pi_K(z- z_I) )_{ K}  + ( (\beta_h - \beta) \nabla  u, \nabla \Pi_K(z- z_I) )_{ K}.
 \end{split}
\end{equation*}
By the projection, we have $( \beta_h (\nabla \Pi_K u_h - \nabla u)), \nabla \Pi_K(z- z_I) )_{ K}  = ( \beta_h \nabla \Pi_K ( u_h -  u)), \nabla \Pi_K(z- z_I) )_{ K} $.
Applying integration by parts to $\dagger_1$ and $(\beta \nabla u, \nabla \Pi_K( z-z_I) )$, and using the scheme \eqref{eq_VEM} for $( \beta_h \nabla \Pi_K u_h, \nabla \Pi_K z_I )_{ K}$, we arrive at
\begin{equation*}
\begin{split}
\label{thm_u_eq6_2}
\dagger_1 - \dagger_2& = (f, z)_K - (f,\Pi_Kz_I) + S_K(u_h - \Pi_Ku_h, z_I - \Pi_Kz_I) -  ( \beta_h \nabla \Pi_K (u_h-u), \nabla \Pi_K(z- z_I) )_{ K}  \\ 
&- ( f,  \Pi_K(z- z_I) )_{ K} - (\beta \nabla u\cdot\bfn, \Pi_K(z-z_I))_{\partial K}  - ( (\beta_h-\beta) \nabla u, \nabla \Pi_K(z- z_I) )_{ K}  \\
& = (f, z - \Pi_K z)_K + S_K(u_h - \Pi_Ku_h, z_I - \Pi_Kz_I)  -  ( \beta_h \nabla \Pi_K (u_h-u), \nabla \Pi_K(z- z_I) )_{ K} \\
& -  (\beta \nabla u\cdot\bfn, \Pi_K(z-z_I))_{\partial K}  - ( (\beta_h-\beta) \nabla u, \nabla \Pi_K(z- z_I) )_{ K}  
\end{split}
\end{equation*}
where we have merged all the terms involving $f$ in the second equality.
Then, the terms in the right-hand side of equality above after being summed over all the elements are denoted as $\ddag_1$,..., $\ddag_5$, respectively. 
The following estimates are immediately given by the previously-established results:
\begin{subequations}
\begin{align*}
\ddag_1 &\lesssim h^2 \| f \|_{0,\Omega} \| z \|_{2,\Omega}, &\text{(by \eqref{L2_assump})} \\
 \ddag_2 &\le \sum_{K\in\mathcal{T}_h} \| u_h - \Pi_K u_h \|_{S_K}  \| z_I - \Pi_K z_I \|_{S_K}   \lesssim h^2 \| u \|_{2,\Omega} \| z \|_{2,\Omega}, &\text{(by \eqref{cor_energy_eq0} and \eqref{cor_uI_eq0})} \nonumber \\
 \ddag_3 & \lesssim  \sum_{K\in\mathcal{T}_h}  \| \nabla \Pi_K(u-u_h) \|_{0,K}  \| \nabla \Pi_K(z- z_I) \|_{0,K}  \lesssim h^{2} \| u \|_{2,\Omega} \| z \|_{2,\Omega}, & \text{(by \eqref{thm_energy_est_eq0} and \eqref{approxi_VK_Pi})}  \nonumber \\
\ddag_4 & \lesssim h^{1/2} \| u \|_{2,\Omega} \left( \sum_{K\in\mathcal{T}_h} \| (z - z_I) - \Pi_K (z-z_I) \|^2_{0, \partial K}  \right)^{1/2} \lesssim h^{2} \| u \|_{2,\Omega} \| z \|_{2,\Omega}, & \text{(by \eqref{L2_assump}, \eqref{cor_uI_eq0}) and \eqref{lem_est_partialK_eq0}} \nonumber \\
 \ddag_5 & \lesssim h \|  u \|_{2, \Omega} \left( \sum_{K\in\mathcal{T}_h} \| \nabla \Pi_K(z- z_I) \|^2_{0,K} \right)^{1/2} \lesssim h^{2} \| u \|_{2,\Omega} \| z \|_{2,\Omega} . & \text{(by \eqref{betah_approx} and \eqref{approxi_VK_Pi})} 
\end{align*}
\end{subequations}
Putting these estimates into \eqref{thm_u_eq5} leads to the estimate for $(Ic)$, which is combined with $(Ia)$ and $(Ib)$ to conclude the estimate for $(I)$. 
In addition, the estimation of $(II)$ is similar to $\ddag_4$ above. 
$(I)$ and $(II)$ together lead to the estimate of $\Pi(u-u_h)$ by the elliptic regularity $\|z \|_{2,\Omega}\lesssim \| \Pi(u-u_h) \|_{0,\Omega}$. 
Then, the proof is finished by applying triangular inequality to $u - \Pi_K u_h = (u - \Pi_Ku) + (\Pi_Ku_h - \Pi_K u)$ with Hypothesis \hyperref[asp:H5]{(H4)}.
\end{proof}

\begin{remark}
\label{rem_highorder}
It is possible to generalize the general discussion in this section to the high-order case,
where one of the major modification is the boundary space that should include high-order 2D FE spaces.
However, the major challenge is to theoretically verify these hypotheses.
We also point out that verifying these hypotheses is also the main difficulty in this work.
\end{remark}


 \section{Application I: elements with non-shrinking inscribed balls}
 \label{sec:fitted_1}

In this section, we consider the anisotropic meshes in Case (1): elements are allowed to merely contain \textit{but not necessarily star convex to} non-shrinking balls.
It typically arises from fitted meshes, and thus we shall let $\beta=\beta_h=1$ to facilitate a clear presentation. 
In this case, we let $\mathcal{W}_h(K) = \mathcal{P}_1(K)$ and thus $\text{Tr}_{\partial K} ~ \mathcal{W}_h(K) \subseteq  \mathcal{B}_h(\partial K)$.
Then, $\nabla\Pi_K\cdot$ is just the standard $L^2$ projection to the constant vector space. 
All these setups are widely employed in the VEM literature. 
Hypothesis \hyperref[asp:H2]{(H1)} has been discussed in Section \ref{subsec:stab} and given by Lemma \ref{lem_stab_f_verify} under Assumption \hyperref[asp:A1]{(A1)} and Assumption \hyperref[asp:A2]{(A2)}, 
 while Hypothesis \hyperref[asp:H6]{(H5)} is trivial. 
We proceed to examine other Hypotheses below. 
It is worth mentioning that the standard interpolation estimates based on the MAC in \cite{1999AcostaRicardo,1976BabuskaAziz} are not directly applicable as it requires higher regularity assumptions, see Remark \ref{rem_interp_maxangle} below.  



We make the following assumption.
\begin{itemize}
    \item[(\textbf{A3})] \label{asp:A3} Each element $K$ contains a ball $B_K$ of the radius $\mathcal{O}(h_K)$,
    i.e., there is a uniform $\lambda\in(0,1)$ such that the radius $\rho_{B_K}\ge \lambda h_K$, $\forall K\in \mathcal{T}_h$. 
    In addition,  there are neighbor elements $K_j$, $j=1,...,r$ such that $\conv(K) \subset \cup_{j=1}^r K_j$ with $r$ uniformly bounded.
\end{itemize}
\begin{remark}
\label{rem_nonshriking}
We highlight that Assumption \hyperref[asp:A3]{(A3)} does not require the star convexity with respect to $B_K$, 
so it is much weaker than the one in \cite{Brenner;Sung:2018Virtual}. 
In addition, it does not require that each face has a supporting height $\mathcal{O}(h_K)$ towards $K$, 
so it is also weaker than \cite{Cao;Chen:2018AnisotropicNC}. See Figure \ref{fig:anisotropic_element} for an example. 
Nevertheless, we point out that $\conv(K)$ is indeed convex with respect to $B_K$.
\end{remark}


Based on this assumption, we have the following trace inequality only for polynomials.
\begin{lemma}[A trace inequality on anisotropic elements]
\label{lem_traceinequa}
Under Assumption \hyperref[asp:A3]{(A3)}, there holds 
$$
\| v_h \|_{0,\partial K} \le  16 \lambda^{-3} \sqrt{\frac{|\mathcal{F}_K|}{3}}   h^{-1/2}_K \| v_h \|_{0,K}, ~~~~ \forall v_h \in \mathcal{P}_1(K).
$$
\end{lemma}
\begin{proof}
Let $B_K$ be the largest inscribed ball of $K$ with the center $O$, and let $\widetilde{B}_K$ be the ball centering at $O$ of the radius $h_K$. 
Clearly, $B_K \subset K\subset \conv(K) \subset \widetilde{B}_K$, 
and $\widetilde{B}_K$ is a homothetic mapping of $B_K$ of the ratio $\rho_{\widetilde{B}_K}/\rho_{{B}_K}\le \lambda^{-1}$ by Assumption \hyperref[asp:A3]{(A3)}. 
By \cite[Lemma 2.1]{2021ChenLiXiang} and \cite[Lemma 2.2]{2016WangXiaoXu}, we have
\begin{equation}
\label{lem_traceinequa_eq1}
\| v_h \|^2_{L^2(\widetilde{B}_K)} \le \left( \frac{1+\sqrt{1-\lambda^2}}{\lambda} \right)^5 \| v_h \|^2_{L^2({B}_K)},
\end{equation}
where the constant can be enlarged to be $(2\lambda^{-1})^5$ for simplicity.
Then, \hyperref[asp:G1]{(G1)} in Lemma \ref{lem_tet_maxangle}, Assumption \hyperref[asp:A3]{(A3)} and \eqref{lem_traceinequa_eq1} yield 
\begin{equation}
  \begin{split}
\| v_h \|_{0,\partial K} & \le \sqrt{\frac{8|\mathcal{F}_K|}{3}} \rho^{-1/2}_{B_K} \| v_h \|_{0,\conv(K)} 
\le  \sqrt{\frac{8|\mathcal{F}_K|}{3\lambda}} h^{-1/2}_K \| v_h \|_{0,\widetilde{B}_K} \\
 & \le 16 \lambda^{-3} \sqrt{\frac{|\mathcal{F}_K|}{3}}  h^{-1/2}_K \| v_h \|_{0,{B}_K}  
 \le  16 \lambda^{-3} \sqrt{\frac{|\mathcal{F}_K|}{3}}  h^{-1/2}_K \| v_h \|_{0,K}.
  \end{split}
\end{equation}
\end{proof}

Next, we estimate the interpolation errors on the boundary $\partial K$.
The following lemma is the stability of interpolation, which usually only hold in 1D.
\begin{lemma}
\label{lem_interp_1D}
Given an edge $e$, let $I_e$ be the 1D interpolation on $e$. Then, $\forall u\in H^1(e)$, there holds
\begin{equation}
\label{lem_interp_1D_eq0}
| I_e u |_{1,e} \le |u|_{1,e}.
\end{equation}
\end{lemma}
\begin{proof}
Let $A_1$ and $A_2$ be the two ending points of $e$. It follows from the H\"older's inequality that
\begin{equation}
\label{lem_interp_1D_eq1}
\| \partial_{\bft_e} I_e u \|^2_{0,e} = |e|^{-1} (u(A_2) - u(A_1))^2 = |e|^{-1} \left( \int_{e} \partial_{\bft_e} u \dd s \right)^2 \le \| \partial_{\bft_e} u \|_{0,e}.
\end{equation}
\end{proof}

\begin{lemma}
\label{lem_uI_face}
Let $u\in H^2(\conv(K))$.
Under Assumptions \hyperref[asp:A1]{(A1)} and \hyperref[asp:A3]{(A3)}, there holds
\begin{subequations}
\label{lem_uI_face_eq0}
\begin{align}
   &   \| u - u_I \|_{0, \partial K} \lesssim h^{3/2}_K \| u \|_{2,\conv(K)}  , \label{lem_uI_face_eq01}  \\ 
    &   | u - u_I |_{1,\partial K} \lesssim h^{1/2}_K \| u \|_{2,\conv(K)} . \label{lem_uI_face_eq02} 
\end{align}
\end{subequations}
\end{lemma}
\begin{proof}
Given an element $K$ and a triangle $T\in\mathcal{T}_h(\partial K)$,
consider the projection $\mathrm{P}^k_{\conv(K)}$, $k=0,1$. For \eqref{lem_uI_face_eq01}, we have
\begin{equation}
\label{lem_uI_face_eq1}
\| u - u_I \|_{0, T} \le \| u - \mathrm{P}^1_{\conv(K)} u \|_{0, T}  +  \| u_I - \mathrm{P}^1_{\conv(K)} u \|_{0, T} .
\end{equation}
The trace inequality with Assumption \hyperref[asp:A1]{(A1)}, Lemma \ref{lem_proj} and \hyperref[asp:G1]{(G1)} in Lemma \ref{lem_tet_maxangle} imply
\begin{equation}
\begin{split}
\label{lem_uI_face_eq2}
\| u - \mathrm{P}^1_{\conv(K)} u \|_{0, T} & \lesssim h^{-1/2}_K \| u - \mathrm{P}^1_{\conv(K)} u \|_{0,\conv(K)} \\ 
&+ h^{1/2}_K  | u - \mathrm{P}^1_{\conv(K)} u |_{1,\conv(K)} \lesssim h^{3/2}_K |u|_{2,\conv(K)}.
\end{split}
\end{equation}
As for the second term in \eqref{lem_uI_face_eq1}, noticing that $u_I - \mathrm{P}^1_{\conv(K)} u\in \mathcal{P}_1(T)$, we trivially have 
\begin{equation}
\label{lem_uI_face_eq3}
\| u_I - \mathrm{P}^1_{\conv(K)} u \|_{0, T} \lesssim |T|^{1/2} \abs{ (u_I - \mathrm{P}^1_{\conv(K)} u )(\bfa) },
\end{equation} 
where $\bfa$ is some vertex of $T$.
We consider the shape-regular tetrahedron $T'$ given by \hyperref[asp:G3]{(G3)} in Lemma \ref{lem_tet_maxangle} that has $\bfa$ as one vertex, and let $I_{T'}$ be the standard Lagrange interpolation on $T'$. Then, by applying the trace inequality and the triangular inequality, we obtain
\begin{equation}
\begin{split}
\label{lem_uI_face_eq4}
 \abs{ (u_I - \mathrm{P}^1_{\conv(K)} u)(\bfa)  }& \lesssim h^{-1}_K \| I_{T'} u - \mathrm{P}^1_{\conv(K)} u \|_{0,T'} \\
 & \lesssim h^{-1}_K \left(  \| I_{T'} u -  u \|_{0,T'} +  \| \mathrm{P}^1_{\conv(K)} u -  u \|_{0,T'}  \right)  \lesssim h_K | u |_{2,\conv(K)}.
 \end{split}
\end{equation}
Putting \eqref{lem_uI_face_eq4} into \eqref{lem_uI_face_eq3} and combining it with \eqref{lem_uI_face_eq2}, we have \eqref{lem_uI_face_eq01}.

Next, we prove \eqref{lem_uI_face_eq02}. 
The triangular inequality yields
\begin{equation}
\label{lem_uI_face_eq5}
| u - u_I |_{1,T} \le \| \nabla_{\partial K}(u - \mathrm{P}^1_{\conv(K)}) u \|_{0,T} + \| \nabla_{\partial K}(\mathrm{P}^1_{\conv(K)} u - u_I) \|_{0,T}.
\end{equation}
The estimate of the first term in the right-hand side in \eqref{lem_uI_face_eq5} follows from the similar argument to \eqref{lem_uI_face_eq2} with the trace inequality. 
We focus on the second term in \eqref{lem_uI_face_eq5}. By Lemma \ref{lem_grad}, we have 
\begin{equation}
\begin{split}
\label{lem_uI_face_eq6}
\| \nabla_{\partial K} ( u_I - \mathrm{P}^1_{\conv(K)} u ) \|_{0,T} \lesssim & \sum_{e\subseteq\partial T} h^{1/2}_T \| \nabla_{\partial K} (u_I - \mathrm{P}^1_{\conv(K)} u )\cdot \bft_e \|_{0,e}.
 \end{split}
\end{equation}
Let us focus on an arbitrary edge $e$ of $T$.
By \hyperref[asp:G2]{(G2)} in Lemma \ref{lem_tet_maxangle}, there is a shape-regular triangle $T'$ that has $e$ as its one edge and a shape-regular pyramid $T''$ that has $T$ as its one face, both being contained in $\conv(K)$ with the size $\mathcal{O}(h_K)$. 
Then, Lemma \ref{lem_interp_1D} and Lemma 2.1 in \cite{Brenner;Sung:2018Virtual} yield
\begin{equation}
\begin{split}
\label{lem_uI_face_eq7}
& \| \partial_{\bft_e}( I_e u -  \mathrm{P}^1_{\conv(K)} u ) \|_{0,e}  = \| \partial_{\bft_e}I_e ( u -  \mathrm{P}^1_{\conv(K)} u ) \|_{0,e} \\
 \lesssim & |  u -  \mathrm{P}^1_{\conv(K)} u  |_{1,e}  
 \lesssim h^{-1/2}_K |  u -  \mathrm{P}^1_{\conv(K)} u  |_{1,T'} + |  u  |_{3/2,T'} \\
  \lesssim & h^{-1}_K |  u -  \mathrm{P}^1_{\conv(K)} u  |_{1,T''} +   |  u -  \mathrm{P}^1_{\conv(K)} u  |_{2,T''} + |  u   |_{2,T''}
    \lesssim  |  u   |_{2,\conv(K)}.
\end{split}
\end{equation}
Noticing $\nabla_{\partial K}\cdot \bft_e = \partial_{\bft_e} \cdot$, and putting \eqref{lem_uI_face_eq6} and \eqref{lem_uI_face_eq7} into \eqref{lem_uI_face_eq5}, we finish the proof.
\end{proof}

\begin{remark}
\label{rem_interp_maxangle}
Interpolation estimates on a triangle $T$ with MAC generally demand relatively higher regularity \cite{1999Duran,1994Shenk}:
$\| \nabla(u- I_T) \|_{0,T} \lesssim h_T \| u \|_{2,T}$.
But on faces $u$ has merely $H^{3/2}$ regularity, this argument cannot be directly applied to obtain the bound in terms of $\|\cdot\|_{2,\conv(K)}$. 
If $T$ is a tetrahedron, then the interpolation estimate requires even higher regularity \cite{1999Duran}, i.e., $W^{p,2}$, $p>2$, which cannot be further improved, see the counterexample in \cite{1994Shenk}. 
This property adds more complexity to the anisotropic analysis for 3D shrinking elements, see the discussion in the next section.
\end{remark}

Now, Hypotheses \hyperref[asp:H3]{(H2)}-\hyperref[asp:H5]{(H4)} follow from the estimates above.
\begin{lemma}
\label{lem_assump_verify1}
Under Assumptions \hyperref[asp:A1]{(A1)} and \hyperref[asp:A3]{(A3)}, Hypothesis \hyperref[asp:H3]{(H2)} holds.
\end{lemma}
\begin{proof}
We need to estimate $\| \nabla \Pi_K( u - u_I) \|_{0,K}$ and $\| ( u - u_I) - \Pi_K( u - u_I) \|_{S_K}$.
Let $\bfp_h = \nabla \Pi_K( u - u_I) \in [\mathcal{P}_0(K)]^3$. Integration by parts and Lemmas \ref{lem_uI_face} and \ref{lem_traceinequa} lead to
$$
\| \bfp_h \|^2_{0,K} = (\bfp_h\cdot\bfn, u-u_I)_{0,\partial K} \le \| \bfp_h \|_{0,\partial K} \| u-u_I \|_{0,\partial K}  \lesssim \| \bfp_h \|_{0,K}  h_K \| u \|_{2,\conv(K)} .
$$
Cancelling one $\| \bfp_h \|_{0,K}$ yields the estimate of $\| \nabla \Pi_K(u-u_I) \|_{0,K}$.
In addition, we note that
\begin{equation}
\label{assump_verify2}
\| (u - u_I) - \Pi_K(u-u_I) \|_{S_K} \le h^{1/2}_K \| \nabla_{\partial K}(u-u_I) \|_{0,\partial K} + h^{1/2}_K \| \nabla_{\partial K} \Pi_K (u-u_I) \|_{0,\partial K}
\end{equation}
of which the first term directly follows from Lemma \ref{lem_uI_face}, and the second term follows from the estimate of $\| \nabla \Pi_K(u-u_I) \|_{0,K}$ and the trace inequality in Lemma \ref{lem_traceinequa}. 
\end{proof}

\begin{lemma}
\label{lem_assump_verify2}
Under Assumptions \hyperref[asp:A1]{(A1)} and \hyperref[asp:A3]{(A3)}, Hypothesis \hyperref[asp:H4]{(H3)} holds.
\end{lemma}
\begin{proof}
Note that \eqref{Pi_approx_eq1} is simple due to $\|\nabla( u - \Pi_{K} u)\|_{0,K} \le \|\nabla( u - \Pi_{\conv(K)} u)\|_{0,K}$ and Lemma \ref{lem_proj}. 
For \eqref{Pi_approx_eq2}, we only need to estimate $\| \nabla(u-\Pi_K u)  \|_{0,\partial K}$ as it can bound both the normal and tangential derivatives.
The triangular inequality yields
\begin{equation}
\label{lem_assump_verify2_eq1}
\| \nabla(u-\Pi_K u) \|_{0,\partial K} \lesssim \| \nabla(u-\Pi_{\conv(K)} u)  \|_{0,\partial K} + \| \nabla(\Pi_{\conv(K)} u - \Pi_K u ) \|_{0,\partial K}.
\end{equation}
The estimate of the first term on the right-hand side follows from the trace inequality with \hyperref[asp:G1]{(G1)} in Lemma \ref{lem_tet_maxangle}, 
while the estimate of the second term follows from Lemmas \ref{lem_traceinequa} and \ref{lem_proj} by inserting $\nabla u$.
\end{proof}

At last, we show Hypothesis \hyperref[asp:H5]{(H4)}.
We highlight the classical argument based on the Poincar\'e inequality, such as (2.15) in \cite{Brenner;Sung:2018Virtual},
is not applicable here, since elements are not shape regular. 
\begin{lemma}
\label{lem_assump_verify3}
Under Assumptions \hyperref[asp:A1]{(A1)} and \hyperref[asp:A3]{(A3)}, Hypothesis \hyperref[asp:H5]{(H4)} holds.
\end{lemma}
\begin{proof}
The second term in \eqref{L2_assump} follows from Lemma \ref{lem_uI_face}. 
We focus on the first term. 
Given any $z\in H^1(K)$, we may write $\Pi_K z = (\mathrm{P}_K^0\nabla z) \cdot(\bfx - \bfx_0) +\mathrm{P}^0_{\partial K} z $ 
where $\bfx_0$ is chosen such that $(\bfx - \bfx_0,1)_{\partial K} = \mathbf{ 0}$. 
We have $\| \bfx - \bfx_0 \| \lesssim h_K$, $\forall\bfx\in K$. Then,  letting $v = \mathrm{P}^1_{\conv(K)} u  - u$, we write
\begin{equation}
\begin{split}
\label{lem_assump_verify3_1}
\| u - \Pi_K u \|_{0,K} & \le \| u - \mathrm{P}^1_{\conv(K)} u \|_{0,K} + \| \mathrm{P}^1_{\conv(K)} u  - \Pi_K u \|_{0,K} \\
& \le \| u - \mathrm{P}^1_{\conv(K)} u \|_{0,K} + \|  (\mathrm{P}_K^0\nabla v) \cdot(\bfx - \bfx_0) \|_{0,K} + \| \mathrm{P}^0_{\partial K} v  \|_{0,K}.
\end{split}
\end{equation}
 The estimates of the first two terms on the right-hand side above directly follow from Lemma \ref{lem_proj}. 
For the last term, the isoperimetric inequality $|K|\lesssim |\partial K|^{3/2}$, 
the trace inequality with \hyperref[asp:G1]{(G1)} in Lemma \ref{lem_tet_maxangle} and Lemma \ref{lem_proj} imply
\begin{equation}
\begin{split}
\label{lem_assump_verify3_2}
\| \mathrm{P}^0_{\partial K} v  \|_{0,K}  & = | K|^{1/2}/|\partial K| \abs{ (v ,1)_{\partial K} }
\le | K|^{1/2}/|\partial K|^{1/2} \| v  \|_{0,\partial K} \\
& \lesssim |\partial K|^{1/4} ( h^{-1/2}_K \| v  \|_{0, \conv(K)} + h^{1/2}_K | v  |_{1, \conv(K)}  ) \lesssim h^2_K \| u \|_{2,\conv(K)}.
\end{split}
\end{equation}
Putting \eqref{lem_assump_verify3_2} into \eqref{lem_assump_verify3_1} finishes the proof.
\end{proof}

As a conclusion, under Assumptions \hyperref[asp:A1]{(A1)}-\hyperref[asp:A3]{(A3)}, Theorems \ref{lem_err_eqn} and \ref{thm_u} give the desired estimate.


\section{Application II: a special class of shrinking elements cut from cuboids}
\label{sec:fitted _2}

In this section, we discuss the anisotropic elements in Case (2) that are cut from a background Cartesian mesh,
given by the following assumption:
\begin{itemize}
    \item[(\textbf{A4})]  \label{asp:A4} 
    Each element $K$ is cut from cuboids by a plane. 
\end{itemize}
Clearly, these elements do not satisfy Assumption \hyperref[asp:A3]{(A3)} in the sense that
the inscribed balls can be arbitrarily small, as shown in Figure \ref{fig:cub},
and they may even shrink to a flat plane or a segment. 
Note that one cannot apply the trace inequality in Lemma \ref{lem_traceinequa}, say on the face $A_1D_1D_3A_5$ of the left plot in Figure \ref{fig:cub}, towards the element.
This very problem will raise significant challenges in analysis, 
which requires a completely different argument from the prevision section, also distinguished from those in the literature.


In this case, by rotation, there are three types of elements highlighted by the red solids in Figure \ref{fig:cub}. By \cite[Proposition 3.2]{2017ChenWeiWen}, all these elements have a boundary triangulation satisfying Assumptions \hyperref[asp:A1]{(A1)} with the maximum angle $144^{\circ}$. 
In addition, they all satisfy Assumption \hyperref[asp:A2]{(A2)}, (the second and third cases even satisfy the stronger Assumption \hyperref[asp:A2plus]{(A2')}).


Denote $R$ by a cuboid with the size $h_R$. 
For an element $K$ cut from $R$, we have $h_K \le  h_R$.
We follow the setup of the previous section by setting $\beta=\beta_h=1$ and $\mathcal{W}_h(K)=\mathcal{P}_1(K)$. 
Similarly, we still only need to examine Hypotheses \hyperref[asp:H3]{(H2)}-\hyperref[asp:H5]{(H4)}.
To avoid redundancy, we merely concentrate on the highlighted element in the left plot in Figure \ref{fig:cub} 
which is a triangular prism and may shrink to the segment $A_1A_5$. 
From the perspective of analysis, this is the most challenging case.  
In the analysis below, we always take $A_1$ as the origin with $A_1D_1$, $A_1D_2$ and $A_1A_5$ being the $x_1$, $x_2$ and $x_3$ axes.
See Figure \ref{fig:crosssec1} for illustration.

\begin{figure}[h]
  \centering
  \includegraphics[width=1.5in]{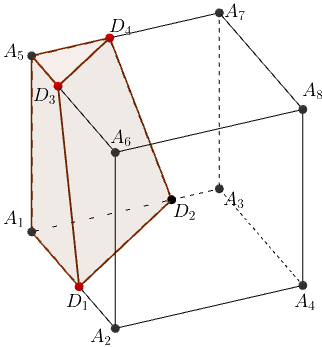}
  \includegraphics[width=1.5in]{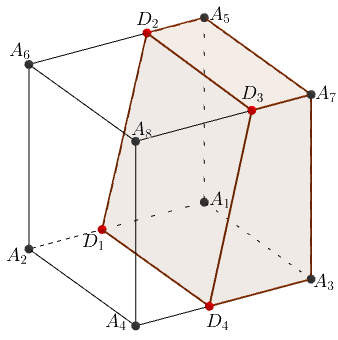}
  \includegraphics[width=1.5in]{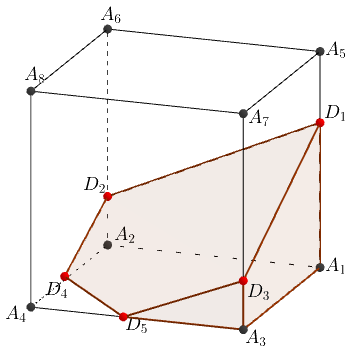}
  \caption{Cases 1-3 from left to right: elements cut from cuboids, highlighted as the red solids, may not contain non-shrinking inscribed balls. 
  The left one may shrink to the segment $A_1A_5$. 
  The middle one and the right one may shrink to the plane $A_1A_3A_7A_5$ and $A_1A_2A_3A_4$. 
  They all satisfy Assumption \hyperref[asp:A1]{(A1)}. Case 1 satisfies Assumption \hyperref[asp:A2]{(A2)}, while Cases 1 and 3 satisfy Assumption \hyperref[asp:A2plus]{(A2')}.
  }
  \label{fig:cub}
\end{figure}


We first present the following lemma which will be frequently used.
\begin{lemma}
\label{lem_edge_est}
Let $u\in H^2(R)$ with $K\subseteq R$. Given each edge $e$ of $K$ with the unit directional vector $\bft_e$, 
there holds that
\begin{equation}
\label{lem_edge_est_eq0}
\abs{ \int_K \nabla(u-u_I)\cdot\bft_e  \dd \bfx } \lesssim h^{1/2}_K |F|^{-1/2} |K| \| u \|_{2,R},
\end{equation}
where $F$ is any face containing $e$.
\end{lemma}
\begin{proof}
First, to facilitate a clear presentation, 
we assume that $A_1$ is the origin and take $A_1D_1$, $A_1D_2$ and $A_1A_5$ to be the $x_1$, $x_2$ and $x_3$ axis, 
see Figure \ref{fig:crosssec1} for illustration of the geometry and notation.
Here, without loss of generality, we always assume $h_{D_1A_1}\ge h_{D_3A_5}$.
To avoid redundancy, we restrict ourselves to one representative case of  $e=A_1A_5$
in the sense that the whole element may shrink to this edge 
and any face containing this edge may not have a supported height of length $\mathcal{O}(h_K)$.
In this case, we can take $F$ to be either one of the two faces $A_1D_1D_3A_5$ and $A_1D_2D_4A_5$ by symmetry
and let $F$ be the face $A_1D_1D_3A_5$ without loss of generality.
The rest of the discussion is technical and lengthy, and thus we shall decompose it into several steps. 

\textit{Step 1. (Rewrite the volume integral as a boundary integral)} 
We rewrite the volume integral in the left-hand side of \eqref{lem_edge_est_eq0} into the integral of some cross-sections parallel to $F=A_1D_1D_3A_5$.
In particular, for each $x_2$, we let $\mathbb{T}(x_2)$ be the cross section at $x_2$ parallel to the plane of $A_1D_1D_3A_5$. 
Note that the cross-sections may be triangular, but it does not affect the analysis below.
Let $\tilde{x}_2$ be the maximal height of $K$ from the $x_1x_3$ plane. 
We can then rewrite the integral as
\begin{equation}
\label{lem_assump_verify4_eq2}
\int_K  \nabla (  u - u_I ) \cdot \bft_e \dd \bfx 
= \int_{0}^{\tilde{x}_2} \left( \int_{\mathbb{T}(x_2)}  \nabla_{\mathbb{T}} (  u - u_I ) \cdot \bft_e \dd x_1 \dd x_3 \right) \dd x_2,
\end{equation}
where $\nabla_{\mathbb{T}}$ is the surface gradient within the plane ${\mathbb{T}}(x_2)$. 
Given a $\mathbb{T}(x_2)$, without loss of generality, assume it is a quadrilateral denoted by $B_1B_2B_3B_4$,
where the edge $B_3B_4$ is parallel to $\bft_e$, i.e., $A_1A_5$. 
As $\bft_e$ is always parallel to the $x_1x_3$ plane, it can be written as $\bft_e = [0,1]^T$ within this plane by dropping ``$0$'' in the second coordinate.
Consider a function $s_h \in \mathcal{P}_1(\mathbb{T}(x_2))$ being sought such that $\rot(s_h) = \bft_e$, where $\rot = [\partial_{x_3}, -\partial_{x_1} ]$ is the 2D rotation operator. In fact, we have $s_h  = - (x_1-\bar{x}_1)$,
where $(\bar{x}_1,\bar{x}_3)$ is the center of $B_1B_2B_3B_4$. 
With integration by parts, we obtain
\begin{equation}
\begin{split}
\label{lem_assump_verify4_eq5}
\abs{ \int_{\mathbb{T}(x_2)}  \nabla_{\mathbb{T}} ( u - u_I ) \cdot \bft_e \dd x_1 \dd x_3 } & 
= \abs{ \int_{\partial\mathbb{T}(x_2)}  \left( \nabla_{\mathbb{T}} (  u - u_I ) \cdot \bft \right) s_h \dd s } \\
 & \le \underbrace{ \|\nabla ( u - u_I)\cdot\bft \|_{0,\partial \mathbb{T}(x_2)} }_{(I)}  \underbrace{\| s_h \|_{0,\partial \mathbb{T}(x_2)}}_{(II)}.
 \end{split}
\end{equation}

\textit{Step 2. (Estimation of the term $(I)$)} There holds
\begin{equation}
\label{lem_assump_verify4_eq6}
\| \nabla ( u - u_I) \cdot \bft \|_{0,\partial \mathbb{T}(x_2)} \le \| (\nabla u - \mathrm{P}^0_R \nabla u )\cdot \bft \|_{0,\partial \mathbb{T}(x_2)}  + \|  (  \mathrm{P}^0_R \nabla u - \nabla u_I)\cdot \bft \|_{0,\partial \mathbb{T}(x_2)} .
\end{equation}
For the first term above, let $e'$ be one edge of $ \mathbb{T}(x_2) = B_1B_2B_3B_4$. 
By Lemma \ref{lem_tet_maxangle}, we can always find a polygon $T\subseteq R$, such that $e'$ is one edge of $T$, 
and $T$ is shape regular in the sense of being star convex to a circle of the radius $\mathcal{O}(h_K)$. 
In addition, we can also find a shape regular polyhedron inside $R$ that has $T$ as its face. 
For example, if $e=B_1B_2$, we let $T = \triangle B_1B_2A_8$.
Then, using the similar argument to \eqref{lem_uI_face_eq7} with Lemma 2.1 in \cite{Brenner;Sung:2018Virtual} and the trace inequality, we have
\begin{equation}
\begin{split}
\label{lem_assump_verify4_eq7}
\| (\nabla u - \mathrm{P}^0_R \nabla u )\cdot \bft \|_{0,e} & \lesssim h^{-1/2}_K \| \nabla u - \mathrm{P}^0_R \nabla u \|_{0,T} + | \nabla u - \mathrm{P}^0_R \nabla u |_{1/2,T}  \lesssim  \| u \|_{2,R}.
\end{split}
\end{equation}
For the second term on the right-hand side of \eqref{lem_assump_verify4_eq6}, 
based on the face triangulation assumption by \hyperref[asp:A1]{(A1)}, 
$\partial \mathbb{T}(x_2)$ is covered by several triangles,
and thus can be decomposed into a collection of edges.
Without loss of generality, we consider one edge $e' = \triangle D_1D_2D_3 \cap \partial \mathbb{T}(x_2)$.
Due to the maximum angle condition, there exist two edges $e_1$ and $e_2$ of $\triangle D_1D_2D_3$ whose angle is bounded below and above.
By Lemma \ref{lem_edge_max}, we have
\begin{equation}
\begin{split}
\label{lem_assump_verify4_eq8}
\|  (  \mathrm{P}^0_R \nabla u - \nabla u_I)\cdot \bft_e \|_{0,e} & \lesssim \sum_{i=1,2} \|  (  \mathrm{P}^0_R \nabla u - \nabla u_I)\cdot \bft_{e_i} \|_{0,e_i} \\
&\lesssim  \sum_{i=1,2} \|  (  \mathrm{P}^0_R \nabla u - \nabla u)\cdot \bft_{e_i} \|_{0,e_i}  +  \sum_{i=1,2} \|  (  \nabla u - \nabla u_I)\cdot \bft_{e_i} \|_{0,e_i} 
\end{split}
\end{equation}
which actually fall into the same scenario as \eqref{lem_assump_verify4_eq7} and  \eqref{lem_uI_face_eq7},
and their estimation are very similar.
Putting \eqref{lem_assump_verify4_eq7} and \eqref{lem_assump_verify4_eq8} into \eqref{lem_assump_verify4_eq6} 
and summing it over all these $e'$, we arrive at
\begin{equation}
\label{lem_assump_verify4_eq8_1}
\| \nabla ( u - u_I) \cdot \bft \|_{0,\partial \mathbb{T}(x_2)} \lesssim \| u \|_{2,R}.
\end{equation}

\textit{Step 3. (Estimation of the term $(II)$)} 
Noticing that $\mathbb{T}(x_2)= B_1B_2B_3B_4$ is convex, we have
\begin{equation}
\label{lem_assump_verify4_eq3}
| s_h (x_1,x_3)| = | x_1 -\bar{x}_1 |  \le l_{B_3B_4}, ~~~~ \forall (x_2,x_3)\in \partial \mathbb{T}(x_1),
\end{equation}
where $l_{B_3B_4}$ is the height of the edge $B_3B_4$ towards $\mathbb{T}(x_2)$, which is in fact $h_{B_1B_4}$,
see the right plot in Figure 5.4 for illustration.
Using \eqref{lem_assump_verify4_eq3}, we have
\begin{equation}
\label{lem_assump_verify4_eq4}
\| s_h \|_{0,\partial \mathbb{T}(x_2) } \lesssim | \partial \mathbb{T}(x_2) |^{1/2} l_{B_3B_4} \lesssim h^{1/2}_K l_{B_3B_4} \lesssim | \mathbb{T}(x_2)| h^{-1/2}_K.
\end{equation}
We point out that the key for constructing such $s_h$ is to get the bound in terms of the possibly shrinking term, 
i.e., $l_{B_1B_2}$, in the second inequality of \eqref{lem_assump_verify4_eq4}.
Now, putting \eqref{lem_assump_verify4_eq8_1} and \eqref{lem_assump_verify4_eq4} into \eqref{lem_assump_verify4_eq5} and \eqref{lem_assump_verify4_eq2}, we have
\begin{equation}
\label{lem_assump_verify4_eq9}
\int_K  \nabla (  u - u_I ) \cdot \bft_e \dd \bfx \lesssim \int_0^{\tilde{x}_2} h^{-1/2}_K \| u \|_{2,R} | \mathbb{T}(x_2)| \dd x_2 \lesssim h^{-1/2}_K |K| \| u \|_{2,R},
\end{equation}
which yields \eqref{lem_edge_est_eq0} as $|F|\lesssim h^2_K$.

\end{proof}

\begin{lemma}
\label{lem_cylin_trace}
Let $u\in H^1(R)$ with $K\subseteq R$. Under Assumptions \hyperref[asp:A1]{(A1)} and \hyperref[asp:A4]{(A4)}, $\forall F\in \mathcal{F}_K$,
\begin{equation}
\label{lem_cylin_trace_eq0}
\| u - \mathrm{P}_K^0 u \|_{0,F} \lesssim h^{1/2}_K \| \nabla u \|_{0,R}.
\end{equation}
\end{lemma}
\begin{proof}
When $F= \triangle A_1D_1D_2$ or $\triangle A_5D_3D_4$, the trace inequality immediately yields the desired result.
So we only need to consider the other three faces whose height towards $K$ may shrink.  
Without loss of generality, we let $F$ be the face $D_1D_2D_3D_4$.
By the triangular inequality, we have
\begin{equation}
\label{lem_cylin_trace_eq1}
\| u - \mathrm{P}_K^0 u \|_{0,F}  \le \| u - \mathrm{P}_{F}^0 u \|_{0,F} + \| \mathrm{P}_{F}^0 u - \mathrm{P}_K^0 u \|_{0,F}.
\end{equation}
Let $P_F$ be the pyramid contained in $R$ that has $F$ as its base and $A_8$ as the apex, and it is easy to see that the height of $F$ is $\mathcal{O}(h_K)$.
Then, by the trace inequality, we obtain
\begin{equation}
\begin{split}
\label{lem_cylin_trace_eq2}
\| u - \mathrm{P}_{F}^0 u \|_{0,F} & \lesssim \| u - \mathrm{P}_{R}^0 u \|_{0,F}  \lesssim h^{-1/2}_K \| u - \mathrm{P}_{R}^0 u \|_{0,P_F}  + h^{1/2}_K | u - \mathrm{P}_{R}^0 u |_{1,P_F} \\
& \lesssim h^{-1/2}_K \| u - \mathrm{P}_{R}^0 u \|_{0,R}  + h^{1/2}_K | u  |_{0,R} \lesssim h^{1/2}_K |u|_{1,R}.
\end{split}
\end{equation}
For the second term in \eqref{lem_cylin_trace_eq1}, by the projection property we have
\begin{equation}
\label{lem_cylin_trace_eq3}
\| \mathrm{P}_{F}^0 u - \mathrm{P}_K^0 u \|_{0,F} = |F|^{1/2}/|K| \abs{ \int_K \mathrm{P}_{F}^0 u -u \dd \bfx } \le  |F|^{1/2}/|K| ^{1/2} \| u - \mathrm{P}_{F}^0 u \|_{0,K}.
\end{equation}
Let $l_F$ be the height of $F$ towards $K$.
Note that it is possible $|F|/|K| \rightarrow \infty$ as $K$ shrinks to $F$ i.e., $l_F \rightarrow 0$.
But we will see that the estimate of $\|  u - \mathrm{P}_F^0 u \|_{0,K}$ can compensate the issue.
Without loss of generality, we let $F$ be on the $x_1x_2$ plane with $D_1D_3$ being the $x_1$ axis, see Figure \ref{fig:prism} for illustration, 
and assume the dihedral angle associated with $D_1D_2$ is not greater than $\pi/2$. 
In this case, by the elementary geometry, we know that every face angle and dihedral angle associated with the vertex $D_1$ is uniformly bounded below and above. 
Let $D_1D_3$, $D_1D_2$ and $D_1A_1$ 
have the directional vectors $\bft_1$, $\bft_2$, $\bft_3$. Then, we can construct an affine mapping 
\begin{equation}
\label{lem_cylin_trace_eq3_1}
 \bfx = \mathfrak{F}( \hat{\bfx}) := \left[  \bft_1, ~  \bft_2, ~  \bft_3  \right] \hat{\bfx} := A \hat{\bfx}.
\end{equation}
Note that $\| A^{-1} \|\lesssim |\text{det}(A)|^{-1} = |(\bft_1\times \bft_2)\cdot\bft_3|^{-1}\lesssim 1$ as every angle is bounded below and above, and $\|A\|\lesssim 1$. 
So the mapping \eqref{lem_cylin_trace_eq3_1} maps to $K$ from a triangular prism $\hat{K}= \mathfrak{F}^{-1}(K)$ with the two triangular faces orthogonal to $\hat{F}=   \mathfrak{F}^{-1}(F)$, as shown in Figure \ref{fig:prism}. 
Note that the prism $\hat{K}$ satisfies the condition of Lemma \ref{lem_poin_anis}. 
Let $v= u - \mathrm{P}_F^0 u$ and $\hat{v}(\hat{\bfx}) := v(\mathfrak{F}( \hat{\bfx}) )$, and then we have 
\begin{equation}
\label{lem_cylin_trace_eq4}
\| v \|_{0,K} \lesssim \| \hat{v} \|_{0,\hat{K}} \lesssim l^{1/2}_{\hat{F}} \| \hat{v} \|_{0,\hat{F}} + l_{\hat{F}} \| \nabla \hat{v} \|_{0,\hat{K}}
\lesssim l^{1/2}_F \| u - \mathrm{P}_F^0 u \|_{0,F} + l_F \| \nabla u \|_{0,K}
\end{equation}
where we have used $l_{\hat{F}}\simeq l_F$. 
Substituting \eqref{lem_cylin_trace_eq4} into \eqref{lem_cylin_trace_eq3} yields
\begin{equation}
\begin{split}
\label{lem_cylin_trace_eq5}
\| \mathrm{P}_{F}^0 u - \mathrm{P}_K^0 u \|_{0,F} & \lesssim  (|F| l_F/|K|)^{1/2} \left(  \| u - \mathrm{P}_F^0 u \|_{0,F} + l^{1/2}_F \| \nabla u \|_{0,K}  \right) \lesssim h_K \| \nabla u \|_{0,R},
\end{split}
\end{equation}
where we have used $|F| l_F \le |K|$ and \eqref{lem_cylin_trace_eq2}.
\end{proof}

\begin{figure}[h]
  \centering
  \begin{minipage}{.28\textwidth}
  \centering
  \includegraphics[width=1.4in]{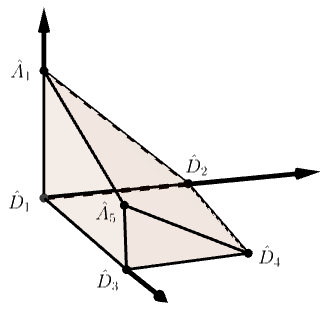}
  \caption{The reference triangular prism $\hat{K}$ mapped to $K$ by $\mathfrak{F}$. The two triangular faces $\hat{D}_1\hat{D}_2\hat{A}_1$ and $\hat{D}_3\hat{D}_4\hat{A}_5$ are perpendicular to the face $\hat{D}_1\hat{D}_3\hat{D}_4\hat{D}_2$}
  \label{fig:prism}
  \end{minipage}
  ~
  \begin{minipage}{0.66\textwidth}
  \centering
   \includegraphics[width=1.6in]{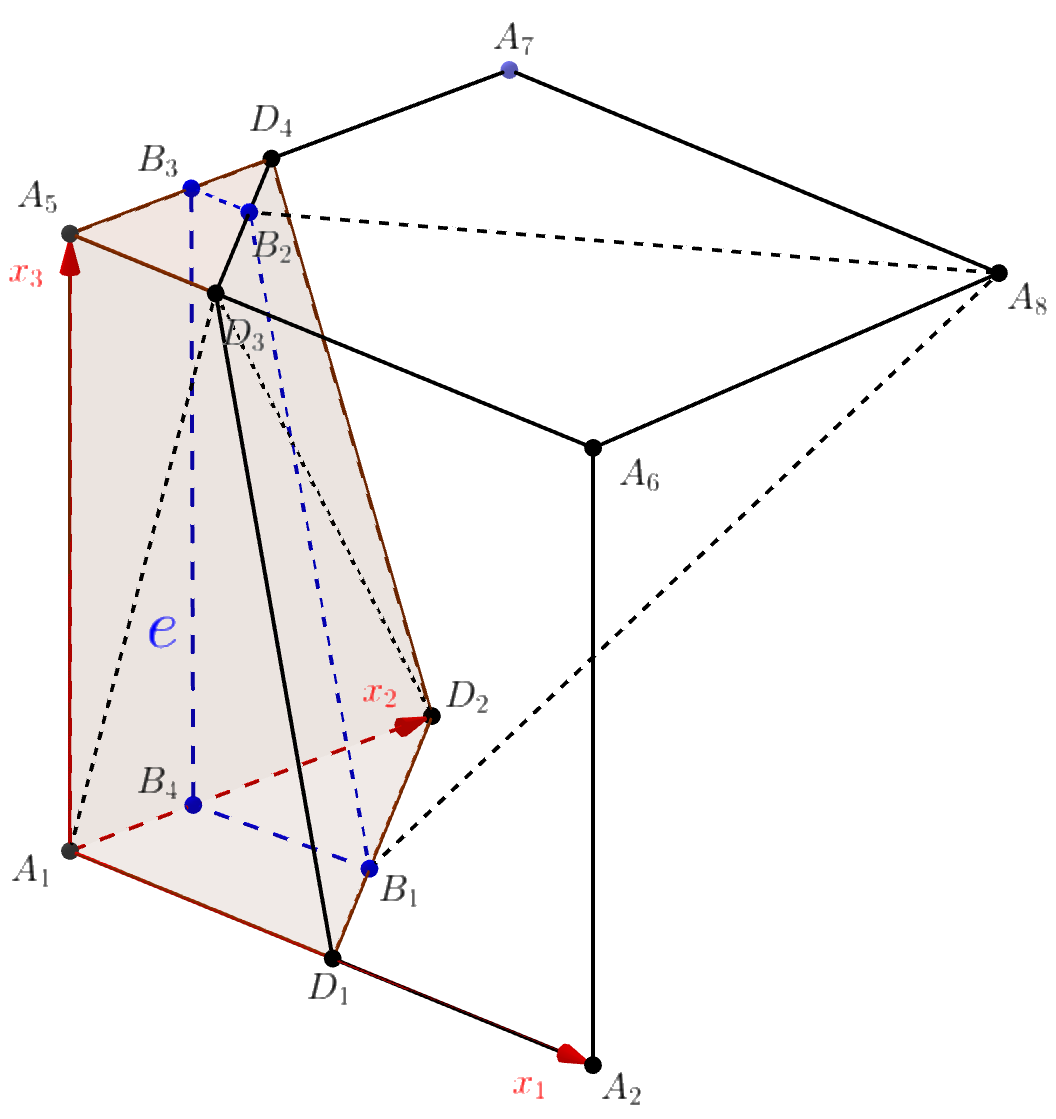}
   \includegraphics[width=2in]{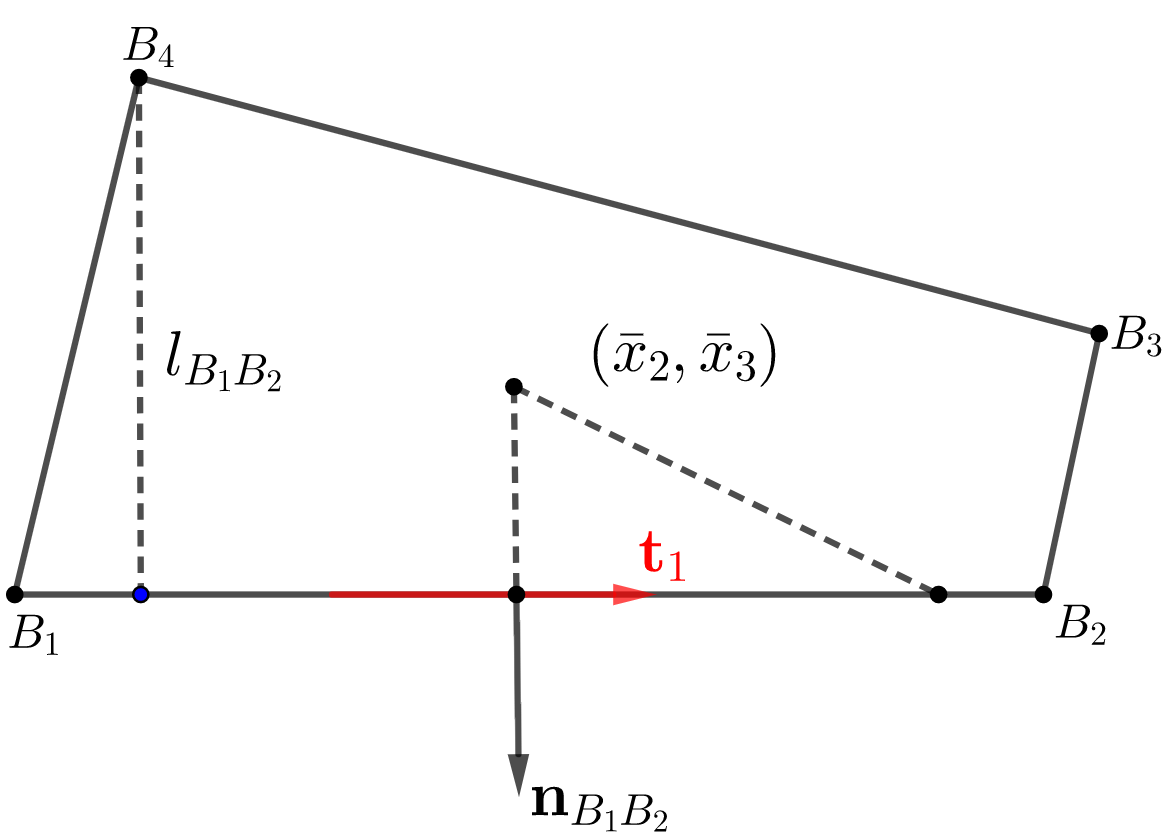}
  \caption{Left: illustration of the proof of Lemma \ref{lem_edge_est} for $e=A_1D_1$. 
  Right: the stage 3 of the proof of Lemma \ref{lem_edge_est}; the estimate of $s_h$ for $\bft_1$, and $|s_h(x_2,x_3)|$ is the length of projection of the vector from $(\bar{x}_1,\bar{x}_2)$ to $(x_1,x_2)$ onto $\bfn_{B_1B_2}$, which cannot be greater than $l_{B_1B_2}$. 
  }
  \label{fig:crosssec1}
  \end{minipage}
\end{figure}

With the preparation above, we are ready to verify Hypotheses \hyperref[asp:H3]{(H2)}-\hyperref[asp:H5]{(H4)}. 

\begin{lemma}
\label{lem_assump_verify4}
Under Assumptions \hyperref[asp:A1]{(A1)} and \hyperref[asp:A4]{(A4)}, Hypothesis \hyperref[asp:H3]{(H2)} holds.
\end{lemma}
\begin{proof}
We first estimate $\|\nabla\Pi_K(u-u_I)\|_{0,K}$. 
Note that there exist three orthogonal edges $e_i$, $i=1,2,3$, of $K$ with the directional vectors $\bft_i$. Using the projection property and Lemma \ref{lem_edge_est}, we have
\begin{equation*}
\begin{split}
\label{lem_uI_face2_eq2}
& \| \nabla \Pi_K ( u - u_I )  \|_{0,K}  = |K|^{1/2} \| \nabla \Pi_K ( u - u_I )  \|  
\lesssim |K|^{1/2} \sum_{i=1,2,3} \| \nabla \Pi_K ( u - u_I ) \cdot\bft_i \| \\
 = & \sum_{i=1,2,3} |K|^{-1/2} \int_K \nabla \Pi_K ( u - u_I ) \cdot\bft_i  \dd \bfx 
 =  \sum_{i=1,2,3} |K|^{-1/2} \int_K \nabla  ( u - u_I ) \cdot\bft_i  \dd \bfx \\
 \le & h^{1/2}_K (|K|/|F|)^{1/2} \| u \|_{2,R} \le h_K \| u \|_{2,R} .
\end{split}
\end{equation*}

For the stabilization term, using the argument similar to Lemma \ref{lem_uI_face}, we have
$ | u - u_I |^2_{1,\partial K} \lesssim h_R \| u \|^2_{2,R}$.
Now, let us concentrate on the estimate of $\| \nabla_{\partial K} \Pi_K( u - u_I ) \|_{0,\partial K}$. 
Given each triangular element $F\in\mathcal{T}_h(\partial K)$, if $F$ is the bottom or top triangular face, i.e., $\triangle A_1D_1AD_2$ and $\triangle A_5D_3AD_4$,
the trace inequality immediately yields the desired result. 
The difficult part is the estimate for the three sided quadrilateral faces where the trace inequality is not applicable due to the possibly shrinking height. 
Based on symmetry, we can assume that $F$ is one triangular element of the quadrilateral face, which, without loss of generality, can be taken as $\triangle A_1D_1D_3$. 
By Lemma \ref{lem_grad}, thanks to MAC, there exist two edges $e_i$, $i=1,2$, with their tangential directional vectors $\bft_i$ such that
\begin{equation*}
\begin{split}
\label{lem_assump_verify4_eq1}
\| \nabla_{F} \Pi_K( u - u_I ) \|_{0,F} & \lesssim \sum_{i=1,2} |F|^{1/2} \| \nabla \Pi_K( u - u_I ) \cdot \bft_i \|  = |F|^{1/2}/|K| \abs{ \sum_{i=1,2} \int_K  \nabla (  u - u_I ) \cdot \bft_i \dd \bfx }.
\end{split}
\end{equation*}
For $F=\triangle A_1D_1D_3$, we can take $\bft_1$ and $\bft_2$ as the directional vectors of $D_1D_3$ and $D_1A_1$. 
Then, Lemma \ref{lem_edge_est} finishes the proof.
\end{proof}

\begin{lemma}
\label{lem_assump_verify5}
Under Assumptions \hyperref[asp:A1]{(A1)} and \hyperref[asp:A4]{(A4)}, Hypothesis \hyperref[asp:H4]{(H3)} holds.
\end{lemma}
\begin{proof}
\eqref{Pi_approx_eq1} is trivial by Lemma \ref{lem_proj}, while \eqref{Pi_approx_eq2} immediately follows from Lemma \ref{lem_cylin_trace}.
\end{proof}

\begin{lemma}
\label{lem_assump_verify6}
Under Assumptions \hyperref[asp:A1]{(A1)}, \hyperref[asp:A2]{(A2)} and \hyperref[asp:A4]{(A4)}, Hypothesis \hyperref[asp:H5]{(H4)} holds.
\end{lemma}
\begin{proof}
The argument is similar to Lemma \ref{lem_assump_verify3}, 
where the difference is to insert $\mathrm{P}^1_K u$ in \eqref{lem_assump_verify3_1} with $v = \mathrm{P}^1_{K} u  - u$
and apply Lemma \ref{lem_cylin_trace} to $\|v \|_{0,\partial K}$ in \eqref{lem_assump_verify3_2}.
\end{proof}

In summary, Assumptions \hyperref[asp:A1]{(A1)} and \hyperref[asp:A4]{(A4)} imply 
Hypotheses \hyperref[asp:H3]{(H2)}-\hyperref[asp:H5]{(H4)},
Hypothesis \hyperref[asp:H2]{(H1)} holds for this type of meshes,
and Hypothesis \hyperref[asp:H6]{(H5)} is trivial.
These results together yield the desired optimal estimates.


\begin{remark}
\label{rem_tetangle}
The proof of Lemma \ref{lem_assump_verify4} heavily relies on that edges of faces cannot be nearly collinear, and respectively, edges of elements cannot be nearly coplanar. 
In fact, this is the essential meaning of the MAC, see Lemma \ref{lem_edge_max} for illustration. 
For the elements in this section, we can find three orthogonal edges. 
For the elements in Section \ref{sec:fitted_1}, Assumption \hyperref[asp:A3]{(A3)} also implicitly implies the existence of such three edges. 
\end{remark}



 \section{Application III: virtual spaces with discontinuous coefficients}
 \label{sec:IVEM}
 
In this section, we consider the third case regarding interface problems on unfitted meshes;
namely one element may contain multiple materials corresponding to different PDE coefficients. 
Henceforth, we assume $\Omega$ is partitioned into two subdomains $\Omega^{\pm}$ by a surface $\Gamma$, called interface. 
We further assume $\beta$ in the model problem \eqref{model} is a piecewise constant function: $\beta|_{\Omega^{\pm}} = \beta^{\pm}$,
where the assumption of two subdomains (two materials) is only made for simplicity. Define $v^{\pm}:=v|_{\Omega^{\pm}}$ for any appropriate function $v$. 
Notice that the regularity $u \in H^1(\Omega)$ and $\beta \nabla u \in \bfH(\text{div};\Omega)$ is not trivial now, 
but are derived from the jump conditions:
\begin{equation}
\label{jump_cond}
\jump{u}_{\Gamma}:=u^+|_{\Gamma} - u^-|_{\Gamma} = 0 ~~~~~ \text{and} ~~~~~  \jump{\nabla u \cdot \bfn}_{\Gamma}:= \nabla u^+|_{\Gamma} \cdot \bfn - \nabla u^-|_{\Gamma} \cdot \bfn = 0. 
\end{equation}
Generally, the jump conditions make the solutions to interface problems only have a piecewise higher regularity. 
Given $D\subseteq \Omega$ intersecting $\Gamma$, define $H^k(D^-\cup D^+) := \{ u\in L^2(D): ~ u|_{D^{\pm}} \in H^k(D) \}$, with an integer $k \ge 0$. For smooth $\Gamma$ and $\partial \Omega$, by \cite{1998ChenZou} the solution $u$ belongs to the space
\begin{equation}
  \label{beta_space_2}
H^2_0(\beta;\Omega) = \{ u\in H^1_0(\Omega) \cap H^2(\Omega^-\cup\Omega^+): \beta\nabla u\in \bfH(\text{div};\Omega) \}
\end{equation}
which is slightly modified from \eqref{beta_space_1} due to the discontinuity of $\beta$.
Define the Sobolev extensions $u^{\pm}_E\in H^2_0(\Omega)$ of $u^{\pm}$ from $\Omega^{\pm}$ to $\Omega$. The following boundedness holds \cite{2001GilbargTrudinger}
\begin{equation}
\label{sobolev_ext}
\| u^{\pm}_E \|_{H^2(\Omega)} \le C_{\Omega}  \| u^{\pm} \|_{H^2(\Omega^{\pm})},
\end{equation}
for a constant $C_{\Omega} $ only depending on the geometry of $\Gamma$. 
We further define the norms $\|u\|_{E,k,D} = \|u^+_E\|_{k,D}+\|u^-_E\|_{k,D}$ and $|u|_{E,k,D} = |u^+_E|_{k,D}+|u^-_E|_{k,D}$ to simplify the presentation.


In this case, the local PDEs in \eqref{lifting} to define the virtual spaces involve discontinuous coefficients.
In fact, \eqref{lifting} contains local interface problems whose solutions belonging to $H^1(K)$ \cite{1998ChenZou,2002HuangZou} satisfying the jump conditions in \eqref{jump_cond} on $\Gamma^K:=\Gamma \cap K$. 
The method analyzed in this section is referred to as the immersed virtual element method (IVEM) developed in \cite{2021CaoChenGuoIVEM,2022CaoChenGuo}
-- a novel immersed scheme for solving interface problems on unfitted meshes, 
different from the classical immersed finite element methods (IFEs) \cite{2020GuoLinZou,2022JiWangChenLi,2015LinLinZhang}.
Thanks to the unfitted meshes, we shall assume the background mesh cut by the interface is a simple tetrahedral mesh, see Figure \ref{fig:tetinter} for an illustration, which is is added into \hyperref[asp:A5]{(A5)} below. 
In fact, this is also a widely-used setup in practice, as meshes are not needed to fit the interface \cite{2020AdjeridBabukaGuoLin,2020GuoLin,2020GuoZhang}.

\begin{figure}[h]
\centering
\begin{minipage}{.4\textwidth}
  \centering
  \includegraphics[width=1in]{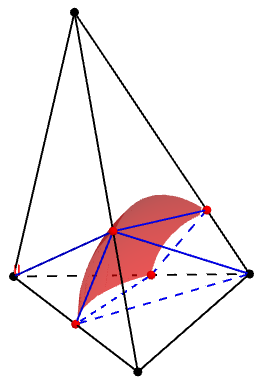}
  \caption{Illustration of an interface element. Being a polyhedron, it has 8 vertices of which 4 are the vertices of $K$ and 4 are the intersecting points with $\Gamma^K_h$. The face triangulation satisfies the }
  \label{fig:tetinter}
\end{minipage}
\begin{minipage}{0.5\textwidth}
  \centering
   \includegraphics[width=2in]{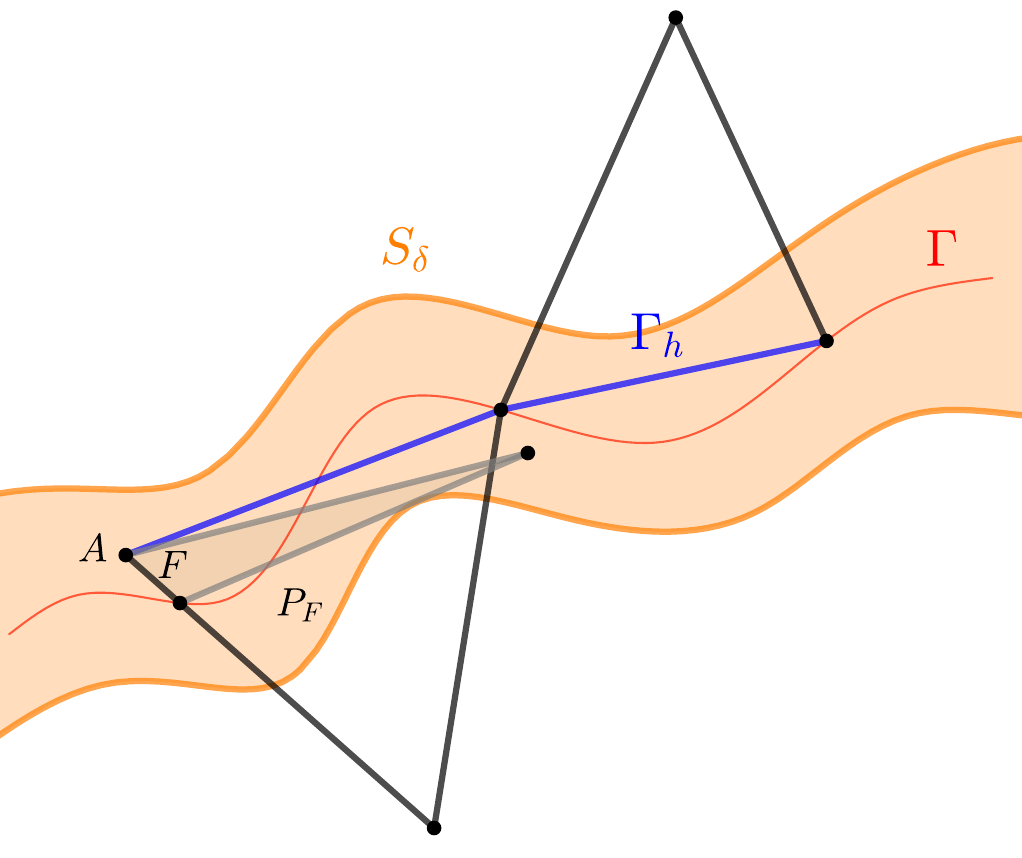}
  \caption{2D Illustration of the approximated interface $\Gamma_h$ and Assumption \hyperref[asp:A5]{(A5)}. }
  \label{fig:assump3}
  \end{minipage}
\end{figure}

Now, let us first review some fundamental ingredients of the IVEM. 
Denote the collection of interface elements: $\mathcal{T}^i_h=\{ K\in \mathcal{T}_h: K \cap \Gamma \neq \emptyset \}$. 
As the linear method is used, we let $\Gamma_h$ be a linear approximation to $\Gamma$, where $\Gamma_h$ can be constructed as a linear interpolant of the level-set function of $\Gamma$ on the mesh $\mathcal{T}_h$,
see Figure \ref{fig:assump3} for a 2D illustration.
Let $\beta_h$ be defined with $\Gamma_h$.
For each $K\in \mathcal{T}^i_h$ intersecting the interface, we let $\Gamma^K= \Gamma \cap K$ and  $\Gamma^K_h= \Gamma_h \cap K$. 
Then, $K$ can be regarded as a polyhedron whose vertices include the vertices of both $K$ and $\Gamma^K_h$, and short edges and small faces may appear. For $\mathcal{T}_h$ being shape-regular, 
it is easy to see that the face triangulation in Assumption \hyperref[asp:A1]{(A1)} holds, 
and thus we still use $\mathcal{B}_h(\partial K)$ as the trace space. 
We refer readers to Figure \ref{fig:tetinter} for illustration of polyhedron and triangulation.

Apparently, polynomial spaces cannot capture the jump behavior across $\Gamma^K_h$ 
and thus Hypothesis \hyperref[asp:H4]{(H3)} certainly does not hold. 
Instead, we employ the following linear IFE spaces as the projection space. 
A local linear IFE function is a piecewise polynomial space given by
\begin{equation}
\label{ife_fun_space}
\mathcal{W}_h(K) := \{ v_h|_{K^{\pm}_h}\in \mathcal{P}_1(K^{\pm}_h): v_h\in H^1(K), ~~~ \beta_h \nabla v_h \in \bfH(\text{div};K) \},
\end{equation}
of which the conditions are equivalent to $\jump{v_h}_{\Gamma^K_h}=0$ and $\jump{\beta_h\nabla v_h\cdot\bar{\bfn}_K}_{\Gamma^K_h} = 0$. 
We then proceed to derive explicit representation of the IFE functions.
The continuity condition shows that $\nabla v_h$ must be continuous tangentially on $\Gamma^K_h$. 
Namely, for $\bar{\bft}^1_K$ and $\bar{\bft}^2_K$ being two orthogonal unit tangential vectors of $\Gamma^K_h$, there holds $\nabla v^-_h \cdot \bar{\bft}^i_K = \nabla v^+_h \cdot \bar{\bft}^i_K$, $i=1,2$.
With the flux jump condition, we have the following identities:
\begin{equation}
\label{gradv}
\nabla v^-_h = M^- \nabla v^+_h ~~~ \text{and} ~~~ \nabla v^+_h = M^+ \nabla v^-_h, ~~~~ M^- = [\bar{\bft}^1_K,\bar{\bft}^2_K,\beta^-\bar{\bfn}_K]^{-T}[\bar{\bft}^1_K,\bar{\bft}^2_K,\beta^+\bar{\bfn}_K]^T,
\end{equation}
where $M^+=(M^-)^{-1}$.
Define the piecewise constant vector space:
\begin{equation}
\label{Cspace}
\bfP^{\beta}_h(K) = \{  \bfp_h |_{K^{\pm}_h}\in[\mathcal{P}_1(K^{\pm}_h)]^3,~ \bfp_h|_{K^-_h} = M^-\bfp_h|_{K^+_h} \}.
\end{equation}
Therefore, given any point $\bfx_K\in \Gamma^K_h$, the IFE space in \eqref{ife_fun_space} is equivalent to
\begin{equation}
\label{lem_ife_space_eq0}
\mathcal{W}_h(K) = \{ \bfp_h \cdot(\bfx - \bfx_K) + c~:~ \bfp_h \in \bfP^{\beta}_h(K), ~ c\in \mathcal{P}_0(K) \}.
\end{equation}
One can easily verify that the space in \eqref{lem_ife_space_eq0} is invariant with respect to the choice of $\bfx_K$
and $\mathcal{P}_0(K) \subseteq \mathcal{W}_h(K)$.
It is trivial that $\text{div}(\beta_h \nabla v_h)=0$, $\forall v_h \in \mathcal{W}_h(K)$. 
Then, the projection in \eqref{Pi_proj} is computable, which is $\beta_h$-weighted and thus different from the standard projection in the previous two cases.


As for the Hypotheses, \hyperref[asp:H2]{(H1)} is guaranteed by Assumptions \hyperref[asp:A1]{(A1)} and \hyperref[asp:A2]{(A2)} again.
We then only need to verify Hypotheses \hyperref[asp:H3]{(H2)}-\hyperref[asp:H6]{(H5)}, 
where we should replace the right-hand sides $\| u \|_{0,\Omega}$ in \eqref{approxi_VK_Pi}-\eqref{betah_approx} by $\| u \|_{0,\Omega^-\cup\Omega^+}$, due to the piecewise regularity. 
Accordingly, the regularity assumptions in Theorems \ref{thm_energy_est} and \ref{thm_u} becomes the space in \eqref{beta_space_2}, 
and the meta-framework developed in Section \ref{sec:unifyframe} is also applicable.
Furthermore, the analysis is standard on non-interface elements, as $\beta_h$ reduces to a constant,
 and thus we focus on interface elements.



Let us first address Hypothesis \hyperref[asp:H6]{(H5)} by considering the geometric error caused by $\Gamma$ and $\Gamma_h$. 
Let $\Gamma_h$ cut $\Omega$ into two polyhedral subdomains $\Omega^{\pm}_h$
that are differing from $\Omega^{\pm}$ in a small region $\widetilde{\Omega}_h := (\Omega^+\cap\Omega^-_h)\cup(\Omega^-\cap\Omega^+_h)$ called the \textit{mismatched region}. 
Define a $\delta$-strip:
$
S_{\delta} = \{ \bfx: \text{dist}(\bfx,\Gamma)\le \delta \}
$ \cite{2010LiMelenkWohlmuthZou}.
Denote $ \widetilde{K} = K\cap \widetilde{\Omega}_h$ and $\widetilde{F} = F\cap \widetilde{\Omega}_h$ for each element $K$ and face $F$.
Make the following assumption:
\begin{itemize}
  \item[(\textbf{A5})] \label{asp:A5} (The $\delta$-strip condition) $\mathcal{T}_h$ is a shape-regular tetrahedral mesh. 
  On $\mathcal{T}_h$, $\Gamma_h$ is an optimal linear approximation to $\Gamma$ in the sense that
  \begin{equation}
\label{delta_strip_arg}
\widetilde{\Omega}_h \subseteq S_{\delta}, ~~~ \text{for some}~ \delta\lesssim h^2.
\end{equation} 
In addition, assume $S_{\delta}$ satisfies that, for each face $F$ of an element $K$, there is a pyramid $P_F\subseteq S_{\delta}\cap\omega_K$ with $\widetilde{F}$ as its base such that the associated supporting height is $\mathcal{O}(h_K)$. 
\end{itemize} 
\eqref{delta_strip_arg} basically means the optimal geometric accuracy of a linear approximation to a surface, which indeed holds for smooth surfaces \cite{2016WangXiaoXu,2020GuoLin}.
A 2D illustration of Assumption \hyperref[asp:A5]{(A5)} is shown in Figure \ref{fig:assump3}. 
Next, we recall the following lemma.
\begin{lemma}{\cite[Lemma 2.1]{2010LiMelenkWohlmuthZou}}
\label{lem_delta}
Let $u\in H^1(\Omega^-\cup\Omega^+)$, then there holds
\begin{equation}
\label{lem_delta_eq1}
\| u \|_{L^2(S_{\delta})} \lesssim \sqrt{\delta} \| u \|_{H^1(\Omega^-\cup\Omega^+)}.
\end{equation}
\end{lemma}
Now, we can control the error occurring in the mismatched region, and show Hypothesis \hyperref[asp:H6]{(H5)}.
\begin{lemma}
\label{lem_h6_verify}
Let $u\in H^2_0(\beta;\Omega)$. Under Assumption \hyperref[asp:A5]{(A5)}, Hypothesis \hyperref[asp:H6]{(H5)} holds.
\end{lemma}
\begin{proof}
It follows from the definition of $\beta_h$ that $\| \beta \nabla u - \beta_h \nabla u \|_{0,K}\lesssim \|\nabla u \|_{0,K\cap S_{\delta}}$,
which yields the estimate in $\|\cdot\|_{0,K}$.
Next, we note that the estimate on faces only appears on those intersecting with the interface. 
Given an interface face $F$, we consider the pyramid $P_F$ from Assumption \hyperref[asp:A5]{(A5)}, by the trace inequality, there holds
\begin{equation}
\begin{split}
\label{lem_h6_verify_eq1}
h^{1/2}_K \| \beta \nabla u\cdot\bfn -  \beta_h \nabla u \cdot\bfn \|_{0,F} & \lesssim \sum_{s=\pm} \| \nabla u^{s}_E\cdot \bfn \|_{0,\widetilde{F}}  \lesssim \sum_{s=\pm}    |u^{s}_E |_{H^1(P_F)} + h_K | u^{s}_E |_{H^2( P_F)} \\
& \lesssim  |u|_{E,1,S_\delta \cap \omega_K}  + h_K |u|_{E,2,S_\delta \cap \omega_K}.
\end{split}
\end{equation}
Summing \eqref{lem_h6_verify_eq1} over all the interface elements and using Lemma \ref{lem_delta} and \eqref{sobolev_ext} finishes the proof.
\end{proof}

We then proceed to verify Hypotheses \hyperref[asp:H3]{(H2)}-\hyperref[asp:H5]{(H4)}.
Note that this is non-trivial as $v_h$ on each interface element is piecewise defined,
and thus all the nice properties for polynomials cannot be applied directly.
In addition, even if the mesh itself is very shape regular, 
each subelement of an interface element could be highly anisotropic.
To handle this issue, let us first recall the following results for the IFE spaces.
\begin{lemma}[Lemma 4.1, \cite{2022CaoChenGuo}]
\label{lem_trace_inequa}
The following trace inequality holds for each $K$:
\begin{equation}
\label{lem_trace_inequa_eq02}
\| \nabla v_h \|_{0, \partial K} \lesssim h^{-1/2}_K \| \nabla v_h \|_{0, K}, ~~~~ \forall v_h\in \mathcal{W}_h(K),
\end{equation}
where the hidden constant is independent of the shape of subelements.
\end{lemma}
As for the interpolation errors, we introduce a specially-designed quasi-interpolation:
\begin{equation}
\label{quaInterp_1}
J_K u =
\begin{cases}
      & J_K^-u = \mathrm{P}^1_{\omega_K} u^-_E, ~~~~~~~~~~~~~~~~~~~~~~~~~~~~~~~~~~~~~~~~~~~~~~~~~~~~~ \text{in} ~ \omega^+_K, \\
      & J_K^+u =\mathrm{P}^1_{\omega_K} u^-_E + (\tilde{\beta} - 1) \nabla \mathrm{P}^1_{\omega_K} u^-_E\cdot\bar{\bfn}_K ( \bfx - \bfx_K )\cdot\bar{\bfn}_K, ~~~ \text{in} ~ \omega^-_K,
\end{cases}
\end{equation}
where $\tilde{\beta} = \beta^-/\beta^+$,
and $\omega_K:= \{ K'\in \mathcal{T}_h: \partial K'\cap K \neq \emptyset \}$ is the patch associated with $K$.
One can easily show $J_K u = \bfp_h\cdot(\bfx-\bfx_K) +c$ with 
$$
\bfp^-_h  = \nabla  \mathrm{P}^1_{\omega_K} u^-_E,
~~~~
\bfp^+_h = M^+ \bfp_h^-, 
~~~~
c =  \mathrm{P}^1_{\omega_K} u^-_E(\bfx_K),
$$
and thus $J_K u$ is an IFE function by \eqref{lem_ife_space_eq0}.
In the following discussion, $J^{\pm}_Ku$ are regarded as polynomials of which each is defined on the entire patch $\omega_K$ instead of just the sub-patches.

\begin{theorem}[Theorem 4.1, \cite{2020GuoLin}]
\label{thm_interp}
Let $u\in H^2_0(\beta;\Omega)$. Then, for every $K\in\mathcal{T}^{i}_h$, there holds
\begin{equation}
\label{thm_interp_eq01}
h^j_K |u^{\pm}_E - J^{\pm}_Ku|_{H^j(\omega_K)} \lesssim h^2_K \|u \|_{E,2,\omega_K}  ,  ~~~~ j=0,1.
\end{equation}
\end{theorem}

With the results above, we are able to estimate the projection errors. 
Similar to $J_K$, due to the piecewise manner,
the projection $\Pi_Ku$ is still a piecewise polynomial.
Again, each of $\Pi^{\pm}_Ku$ is regarded as a polynomial defined on the whole patch. 
The key here is to estimate $u^{\pm}_E - \Pi^{\pm}_K u$ on the whole element.
\begin{lemma}
\label{lem_beta_proj}
Let $u\in H^2_0(\beta;\Omega)$. Then, for every $K\in\mathcal{T}^{i}_h$, there holds
\begin{equation}
\begin{split}
\label{lem_beta_proj_eq0}
\|\nabla(u^{\pm}_E - \Pi^{\pm}_K u) \|_{0, K} & \lesssim  h_K \|u \|_{E,2,\omega_K} + \| u \|_{E,1,\omega_K \cap S_{\delta}}  .
\end{split}
\end{equation}
\end{lemma}
\begin{proof}
 For simplicity, we only show \eqref{lem_beta_proj_eq0} for the ``$-$" component. 
 By the projection property, 
\begin{equation}
\begin{split}
\label{lem_beta_proj_eq1}
\| \sqrt{\beta_h} \nabla(u^-_E - \Pi^{-}_K u) \|_{L^2(K^-_h)}  \le & \| \sqrt{\beta_h} \nabla (u -  \Pi_K u) \|_{0, K} +  \| \sqrt{\beta_h} \nabla u \|_{E,0, K \cap S_{\delta}} \\
 \le & \| \sqrt{\beta_h} \nabla (u -  J_K u) \|_{0, K} +  \| \sqrt{\beta_h} \nabla u \|_{E,0, K \cap S_{\delta}}  \\
  \le & \sum_{s=\pm}  \|\sqrt{\beta^s_h}\nabla (u^s_E -  J^s_K u) \|_{0, K} +   2 \| \sqrt{\beta_h} \nabla u \|_{E,0, K \cap S_{\delta}}
\end{split}
\end{equation}
which yields \eqref{lem_beta_proj_eq0} on $K^-_h$ by Theorem \ref{thm_interp}. As for $K^+_h$, we note that
\begin{equation}
\begin{split}
\label{lem_beta_proj_eq2}
\|\nabla(u^-_E - \Pi^{-}_K u) \|_{0, K^+_h} &\le \| \nabla(u^-_E - J^-_K u) \|_{0, K^+_h} + \|\nabla(\Pi^{-}_K u - J^-_K u) \|_{0, K^+_h}.
\end{split}
\end{equation}
The first term in the right-hand side above directly follows from \eqref{thm_interp_eq01}. 
For the second term, as $v_h:=\Pi_K u - J_K u$ is an IFE function, by \eqref{gradv} and $\| M^+ \| \lesssim 1$, we obtain
\begin{equation}
\begin{split}
\label{lem_beta_proj_eq3}
 \|  \nabla v^-_h \|_{0, K^+_h}  \lesssim \|  \nabla v^+_h \|_{0, K^+_h} 
 \le  \| \nabla(\Pi^{+}_K u - u^+_E) \|_{0, K^+_h} +  \| \nabla( u^+_E - J^+_K u) \|_{0, K^+_h}
\end{split}
\end{equation}
where the estimation of the first term in the right-hand side above is similar to \eqref{lem_beta_proj_eq1} and the estimate of the second term follows from Theorem \ref{thm_interp}. 
Combining these estimates, we obtain \eqref{lem_beta_proj_eq0}.
\end{proof}

With these preparations, we are ready to examine Hypothesis \hyperref[asp:H3]{(H2)}-\hyperref[asp:H5]{(H4)}.

\begin{lemma}
\label{projection term estimate}
Let $u\in H^2_0(\beta;\Omega)$. Under Assumptions \hyperref[asp:A1]{(A1)}, \hyperref[asp:A2]{(A2)} and \hyperref[asp:A5]{(A5)}, 
Hypothesis \hyperref[asp:H3]{(H2)} holds. 
\end{lemma}
\begin{proof}
By the definition of projection and integration by parts, we immediately have
\begin{equation}
\label{projection term estimate eq1}
\begin{aligned}
 \| \sqrt{\beta_h} \nabla \Pi_K (u-u_I)\|_{0, K}^2 =& (\beta_h \nabla \Pi_K (u-u_I)\cdot \mathbf{ n},u-u_I)_{\partial K}\\
 \leq& \| \nabla \Pi_K (u-u_I)\cdot \mathbf{ n}\|_{0, \partial K} \| \beta_h (u-u_I) \|_{0, \partial K}.
\end{aligned}
\end{equation}
As $ \Pi_K (u-u_I)$ is an IFE function, the trace inequality in Lemma \ref{lem_trace_inequa} and Lemma \ref{lem_uI_face} lead to the estimate of $\|  \nabla \Pi_K (u-u_I)\|_{0, K}$ by \eqref{projection term estimate eq1}. 
The estimate of $\| (u-u_I) - \Pi_K(u-u_I) \|_{S_K}$ is similar to \eqref{assump_verify2}.
Summing the estimates over all the interface elements and using Lemma \ref{lem_delta} and \eqref{sobolev_ext} finishes the proof.
\end{proof}

\begin{lemma}
\label{lem_boundflux_err}
Let $u\in H^2_0(\beta;\Omega)$. Under Assumptions \hyperref[asp:A1]{(A1)}, \hyperref[asp:A2]{(A2)} and \hyperref[asp:A5]{(A5)}, 
Hypothesis \hyperref[asp:H4]{(H3)} holds. 
\end{lemma}
\begin{proof}
\eqref{Pi_approx_eq1} immediately follows from Lemma \ref{lem_beta_proj} and Lemma \ref{lem_delta}. 
As for \eqref{Pi_approx_eq2}, by the triangular inequality, given each face $F\in\mathcal{F}_K$, we have
\begin{equation*}
\begin{split}
\label{lem_boundflux_err_eq1}
\|  \nabla (u-\Pi_K u) \|_{0,F}  \le  \sum_{s=\pm}  \|  \nabla (u^{s}_E-\Pi^{s}_K u) \|_{0,F}   +  \|   \nabla (  u^+_E -  u^-_E )\cdot \mathbf{ n}\|_{0,\widetilde{F}}.
\end{split}
\end{equation*}
The estimate of the first term follows from Lemma \ref{lem_beta_proj} with the classical trace inequality applied on the entire element, 
while the estimate of the second term is similar to \eqref{lem_h6_verify_eq1}. 
Summing the estimates over all the interface elements and using Lemma \ref{lem_delta} and \eqref{sobolev_ext} finishes the proof.
\end{proof}

 \begin{lemma}
\label{lem_boundinterp}
Let $u\in H^2_0(\beta;\Omega)$. Under Assumptions \hyperref[asp:A1]{(A1)}, \hyperref[asp:A2]{(A2)} and \hyperref[asp:A5]{(A5)}, 
Hypothesis \hyperref[asp:H5]{(H4)} holds. 
\end{lemma}
\begin{proof}
As $K$ is shape regular and $u,\Pi_K u \in H^1(K)$, by the standard Poincar\'e inequality, there holds $\| u - \Pi_K u \|_{0,K}\lesssim h_K \| \nabla( u - \Pi_K u ) \|_{0,K}$ whose estimate then follows from Lemma \ref{lem_beta_proj}. The estimate of $\| u - u_I \|_{0,\partial K}$ follows from Lemma \ref{lem_uI_face}. Summing the estimates over all the interface elements and using Lemma \ref{lem_delta} and \eqref{sobolev_ext} finishes the proof.
\end{proof}


\section{Numerical results}
\label{sec:num}

In this section, we present numerical results to validate the convergence rates proved above.
In particular, the robustness in terms of the interface location is a well-known issue for methods based on unfitted meshes and fitted anisotropic meshes; namely both the errors and the solver are ideally robust to how the interface cuts the mesh.
For this purpose, we consider such an example: a 3D domain $\Omega = [-1,1]^3$ is partitioned to $N^3$ cuboids, with $N=2^n$, $n=3,4,5,6,7$,
which is used as the background mesh. 
Consider a ``squircle" interface:
\begin{equation}
\label{num_interf}
\Gamma: ~ x_1^4 + x_2^4 + x_3^4 - r_0^4 = 0 ~~~ \text{with} ~~~ r_0 = 0.75-\epsilon,
\end{equation}
which is close to a square but has rounded corner.
Let $\Gamma$ cut the background mesh to generate fitted and unfitted meshes on which both VEMs and IVEMs can be used.
In \eqref{num_interf}, the parameter $\epsilon$ is used to control the interface location relative to the mesh.
For example, if we let $\epsilon$ be very small, say $10^{-6}$,
then there are lots of interface elements within which the interface is very close one of their faces, say Figure \ref{fig:interfmesh} for illustration.
The exact solution is given by
\begin{equation}
\label{exact_solu}
u(x) = 
\begin{cases}
    ( x_1^4 + x_2^4 + x_3^4)^{\alpha}/\beta^-  & \text{in}~ \Omega^-, \\
    ( x_1^4 + x_2^4 + x_3^4)^{\alpha}/\beta^+ + (1/\beta^- - 1/\beta^+)r^{4\alpha}_0  & \text{in}~ \Omega^+, 
\end{cases}
\end{equation}
where $\alpha =1/2$, and all the source terms and boundary conditions are computed accordingly.
\begin{figure}[h]
  \centering
    \includegraphics[width=1.35in]{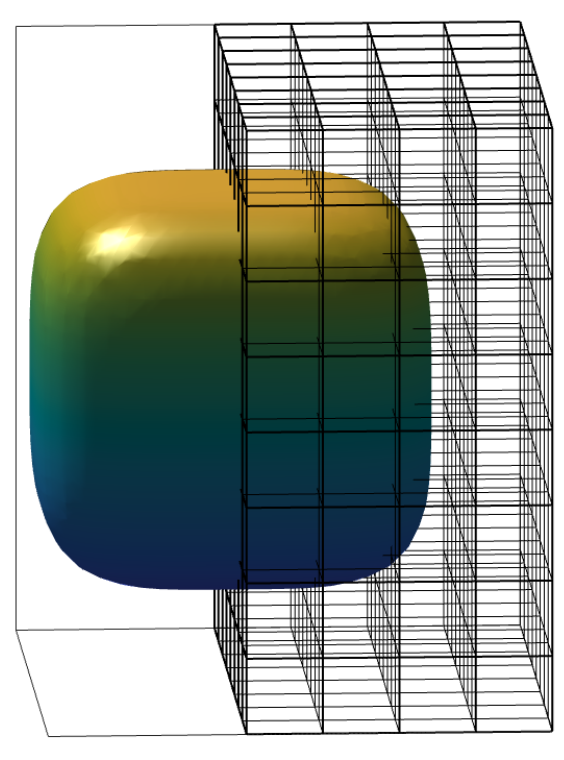}
  ~~
  \includegraphics[width=1.4in]{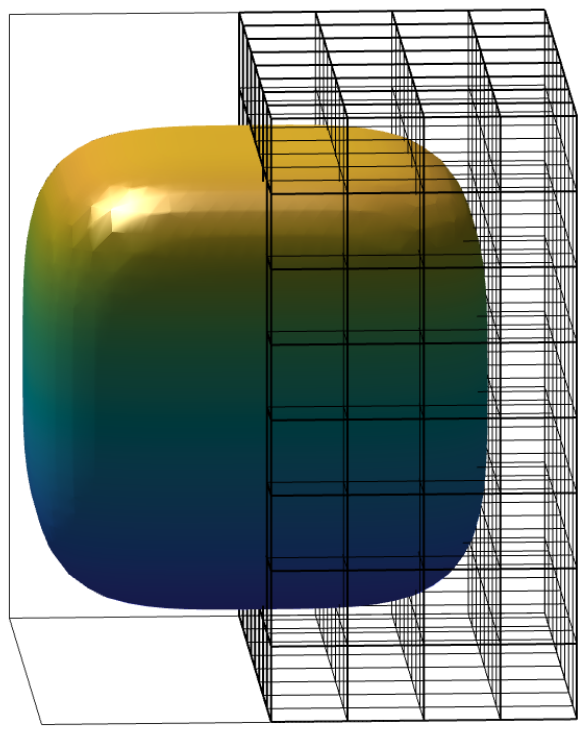}
  ~~
  \includegraphics[width=1.2in]{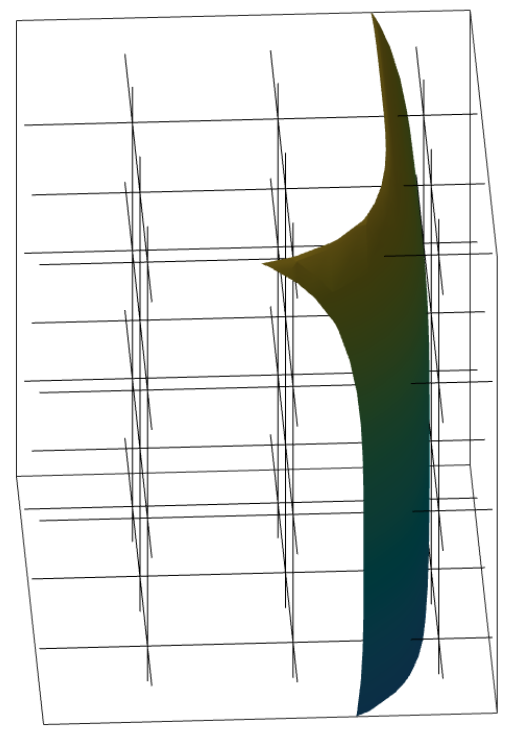}
  \caption{
  The left plot is for $\epsilon = 0.1$, making the interface away from the faces.
  The middle plot i=s $\epsilon = 10^{-6}$, making the interface extremely close the faces with the $x_1$, $x_2$ or $x_3$ coordinates being $\pm 0.75$.
  The third plot is the zoom-in visualization of such interface elements. In particular, it is highlighted that such elements always exist for even very fine meshes.}
  \label{fig:interfmesh}
\end{figure}


We present the numerical results for error convergence and MG iteration numbers, respectively in Figure \ref{fig:convgrate} and Table \ref{table_mgIte}. 
It is clear that both methods exhibit optimal convergence. 
However, the errors for VEMs are generally smaller than those for IVEMs, 
likely because fitted meshes provide greater geometric detail around the interface. 
On the other hand, Table \ref{table_mgIte} highlights a remarkable feature of IVEMs: 
their MG solvers are highly robust to small-cut interface elements, whereas the iteration numbers for VEMs are significantly higher.

\begin{figure}[h]
  \centering
    \includegraphics[width=1.4in]{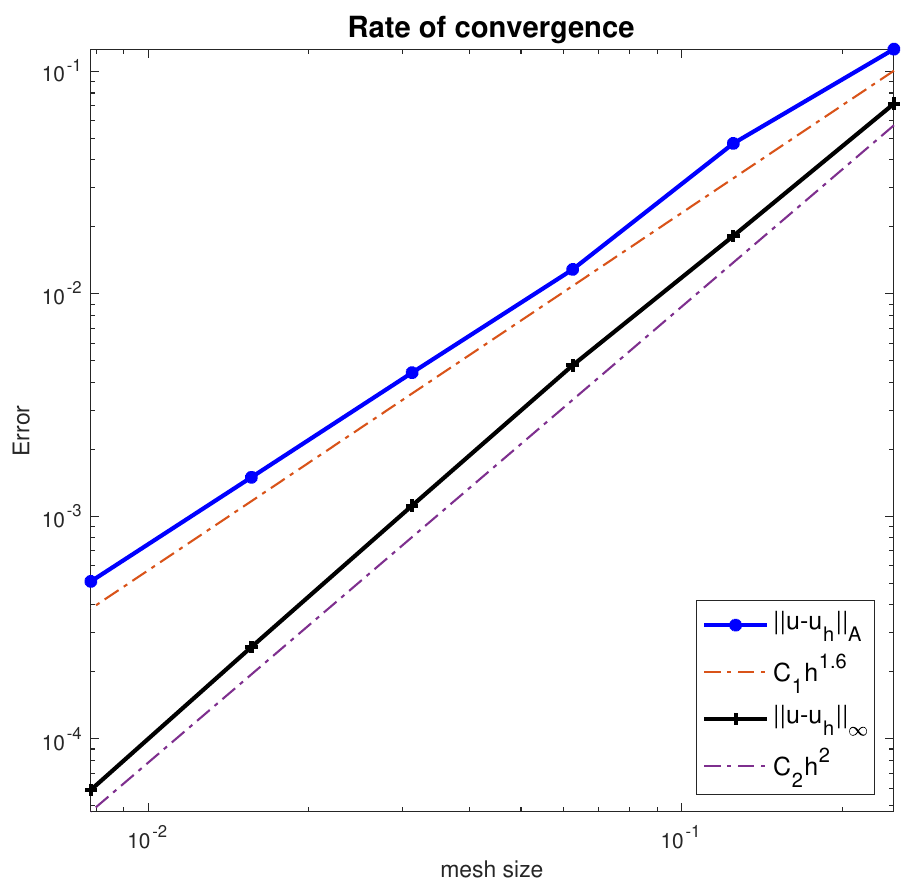}
  ~~
  \includegraphics[width=1.4in]{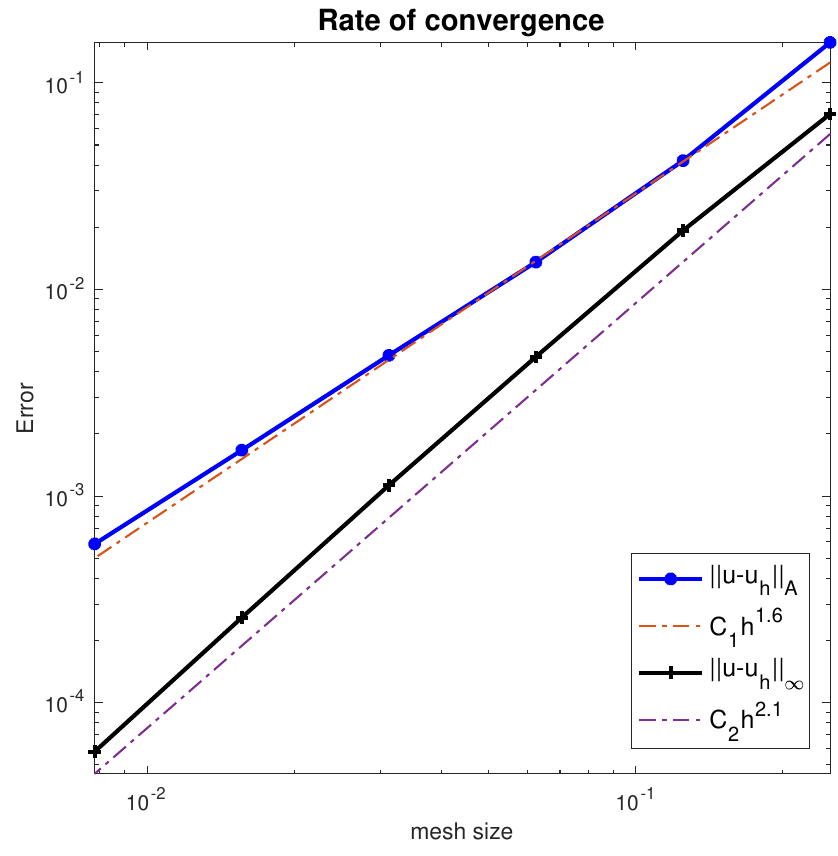}
    \includegraphics[width=1.35in]{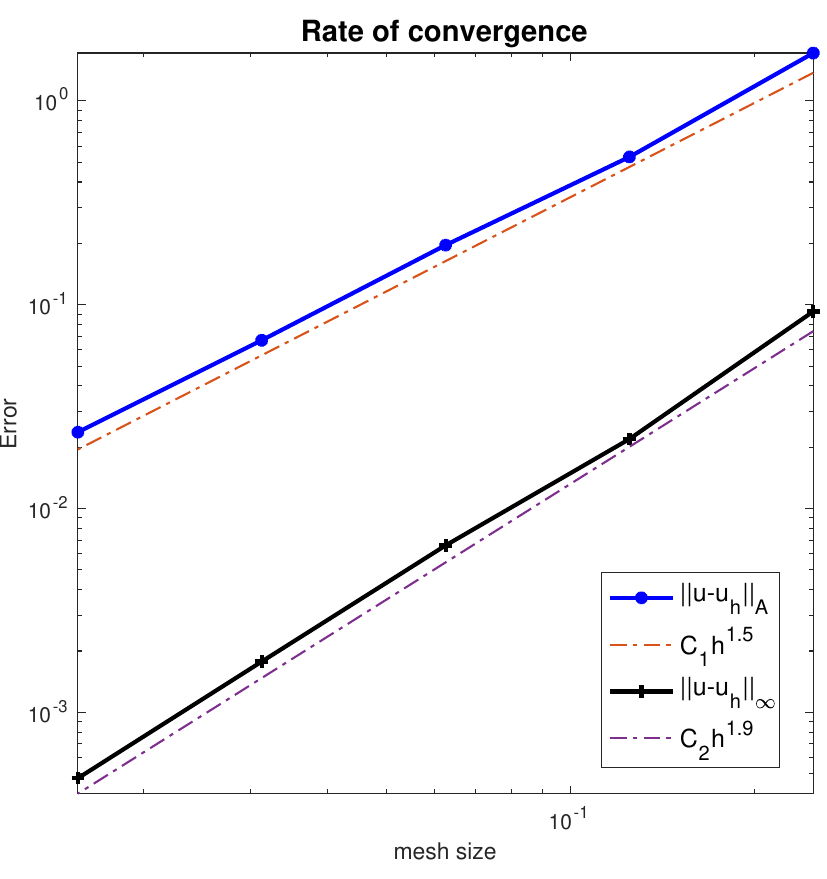}
  ~~
  \includegraphics[width=1.4in]{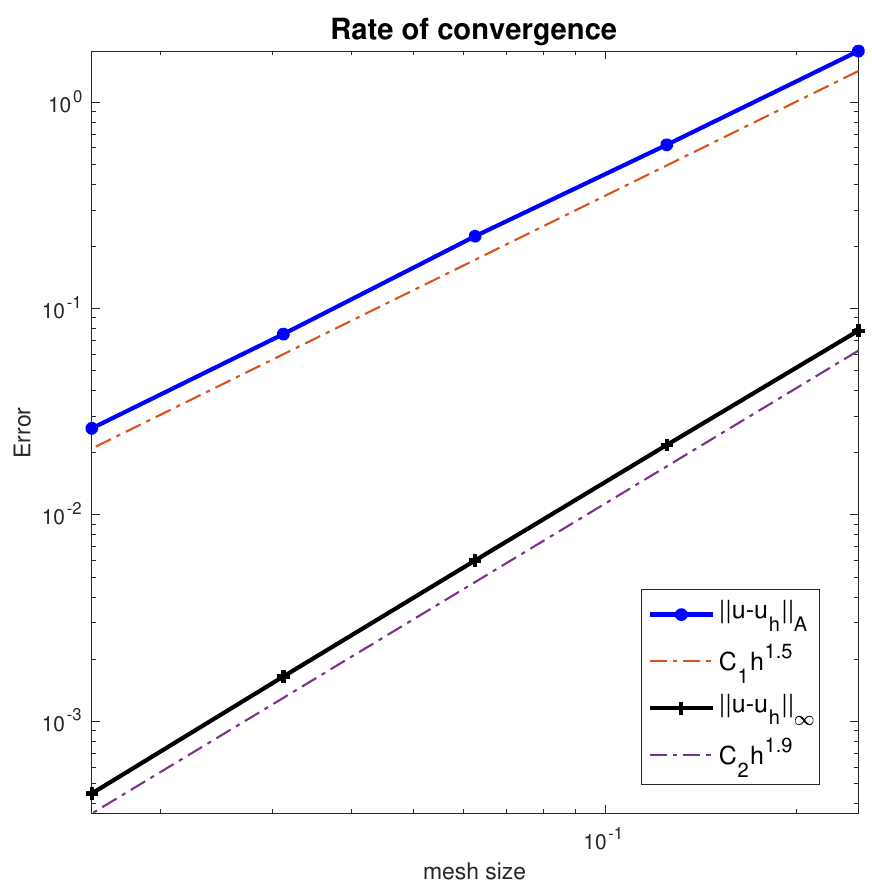}
  \caption{
 The convergence rates for the VEM (the left two, $\epsilon=10^{-1}$ and $\epsilon=10^{-6}$) and the IVEM (the right two, $\epsilon=10^{-1}$ and $\epsilon=10^{-6}$).
 }
  \label{fig:convgrate}
\end{figure}

\begin{table}[h!]
  \centering
  \hspace*{-1cm}
  \begin{tabular}{c | c | c | c | c | c | c | c | c | c | c} 
   \hline
    $N$     &  \multicolumn{2}{c|}{8}   &    \multicolumn{2}{c|}{16}     &    \multicolumn{2}{c|}{32}    &   \multicolumn{2}{c|}{64}    &   \multicolumn{2}{c}{128} \\
   \hline
   $\epsilon$ & $10^{-1}$  & $10^{-6}$  & $10^{-1}$  & $10^{-6}$ & $10^{-1}$  & $10^{-6}$ & $10^{-1}$  & $10^{-6}$ & $10^{-1}$  & $10^{-6}$ \\
   \hline
   VEM & 6 &  7 & 8 & 20 & 9 & 48 & 11 & 200 & 12 & 194 \\ 
   \hline
    IVEM & 7 &  7 & 9 & 10 & 10 & 11 & 12 & 14 &  13 & 17 \\ 
   \hline
  \end{tabular}
  \caption{MG iterations for VEM and IVEM with the two types of interface location.}
  \label{table_mgIte}
  \end{table}


\begin{appendices}

\section{Relation between different geometry assumptions}
\label{appen_lem_tet_maxangle}

\begin{lemma}
\label{lem_tet_maxangle}
Let a polyhedron $D$ be star convex with respect to a ball with the radius $\rho_D$, then the following results hold
\begin{itemize}
\item[(\textbf{G1})]\label{asp:G1} for each $F\in \mathcal{F}_D$, there is a tetrahedron $T\subseteq D$ that has $F$ as of its faces and the supporting height is greater than $\rho_D$.
\item[(\textbf{G2})]\label{asp:G2} for each $e\in \mathcal{E}_D$, there is a trapezoid $T$ that has $e$ as one of its edges and has the largest inscribed ball of the radius larger than $\rho_D/2$. In addition there is a pyramid $T'$ that has $T$ as its base and the height is $\rho_D$
\item[(\textbf{G3})]\label{asp:G3} for each $\bfx \in \mathcal{N}_h(\partial K)$, there is a shape regular tetrahedron $T\subseteq D$ with the size greater than $\rho_D$ that has $\bfx$ as one of its vertices.
\end{itemize}
\end{lemma}
\begin{proof}
Let $O$ be the center of the ball denoted by $B_D$. (\textbf{G1}) can be simply verified by forming a pyramid that has the base $F$ and the apex $O$ as the distance from $O$ to the plane of F is certainly larger than $\rho_D$. For (\textbf{G2}), $T$ can be chosen as the trapezoid formed by $e$ and the segment passing through $O$ parallel to $e$. Consider another point $P$ on $B_D$ such that $PO$ is penperdicular to $T$, then the tetrahedron formed by $T$ and $P$ fulfills the requirement. (\textbf{G3}) follows from a similar argument.
\end{proof}

\begin{figure}[h]
  \centering
  \begin{minipage}{.47\textwidth}
  \centering
  \includegraphics[width=2.3in]{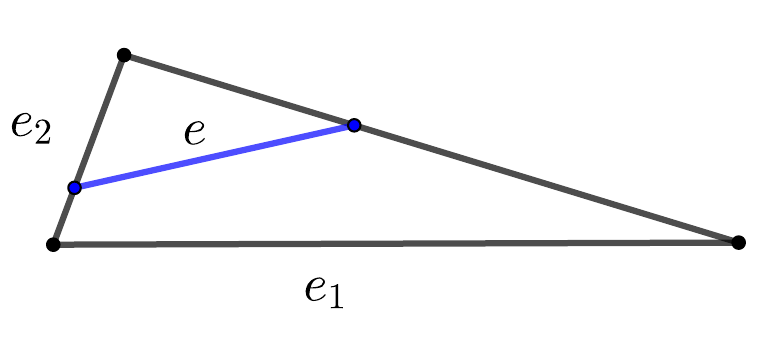}
  \caption{Illustration of the proof of Lemma \ref{lem_edge_max}.}
  \label{fig:ebound}
  \end{minipage}
  ~~~~
    \begin{minipage}{0.47\textwidth}
  \centering
   \includegraphics[width=2.3in]{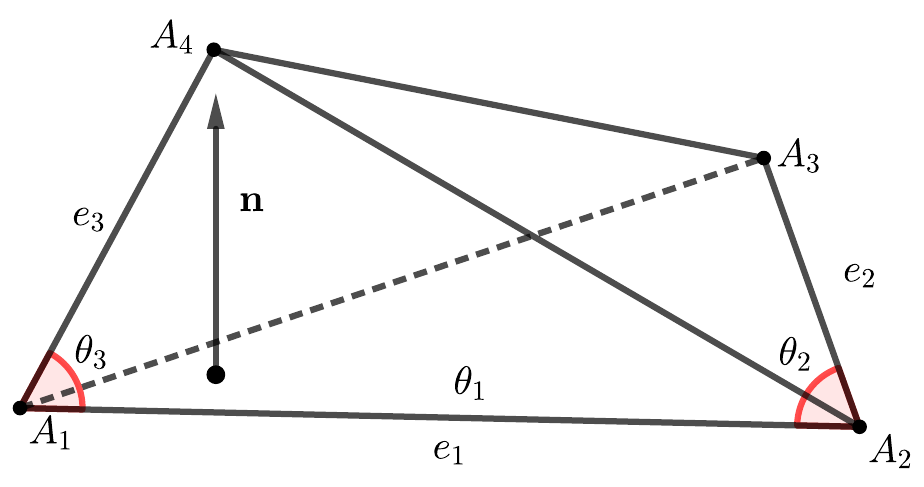}
  \caption{Illustration of the proof of Lemma \ref{lem_regular_edges}.}
  \label{fig:tetra_maxangle}
  \end{minipage}
\end{figure}

\begin{figure}
\centering
   \includegraphics[width=3in]{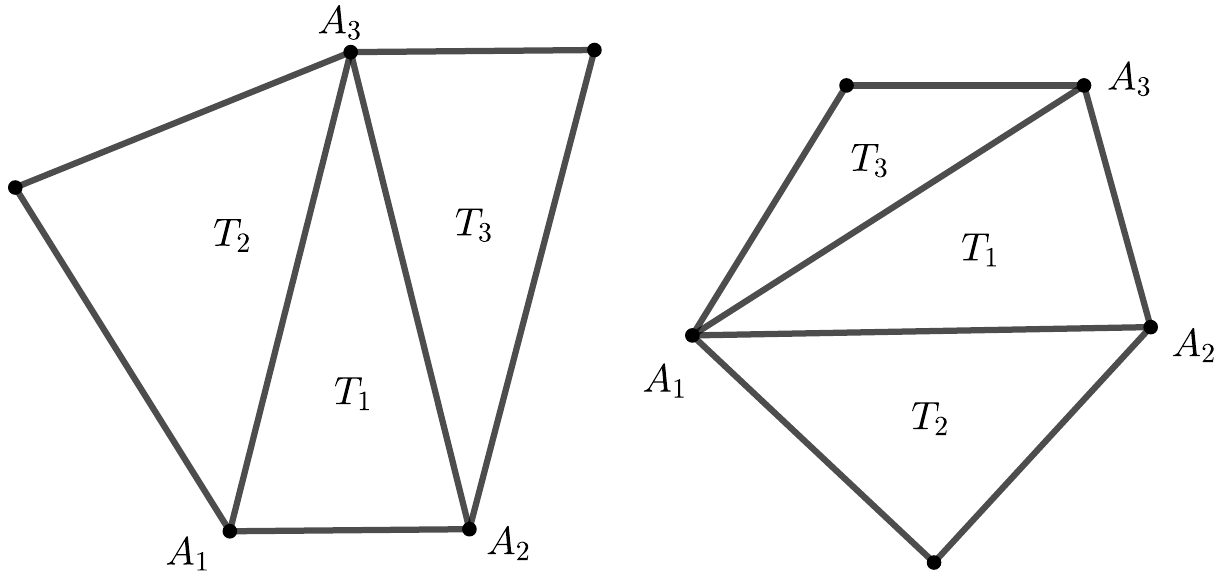}
  \caption{Configuration of anisotropic elements. $T_1$ has the minimum angle at $A_3$ (left) or at $A_1$ (right). Note that these triangles may not be coplanar.}
  \label{fig:anisotrop_triangles}
\end{figure}

\section{An alternative assumption to \hyperref[asp:A2]{(A2)}}
\label{append:lem_A2plus}

\begin{itemize}
  \item[(\textbf{A2}')] \label{asp:A2plus} Every triangle $T$ shares at least one edge with another triangle (including $T$ itself) which satisfies the minimum angle condition and whose size bounded below by $\mathcal{O}(h_T)$.
\end{itemize} 

\begin{lemma}
\label{lem_A2plus}
Assumptions \hyperref[asp:A1]{(A1)} and \hyperref[asp:A2plus]{(A2')} together imply \hyperref[asp:A2]{(A2)} with $\epsilon=\theta_M/\arcsin (\rho \sin(\theta_m) \sin(\theta_M))$, 
where $\theta_M$ is the maximum angle in Assumption \hyperref[asp:A1]{(A1)}, 
$\theta_m$ is the minimum angle in Assumption \hyperref[asp:A2]{(A2)}, 
and $\rho \le 1$ is such that the size of the triangle in Assumption \hyperref[asp:A2]{(A2)} greater than $\rho h_T$.
\end{lemma}
\begin{proof}
We show a stronger version of  Assumption \hyperref[asp:A2]{(A2)} that any two vertices connected by an edge must admit a path satisfying the property in Assumption \hyperref[asp:A2]{(A2)}. 
Call an element \textit{isotropic} if it has the minimum angle $\theta_m$. Consider two vertices $A_1$ and $A_3$ of a triangle $T_1$, as shown in Figure \ref{fig:anisotrop_triangles}.
If $T_1$ is isotropic, then we just choose the path as $A_1$-$A_3$ with $\epsilon = \theta_M/\theta_m$. We focus on $T_1$ being anisotropic.

Case 1. Suppose $\angle A_1A_3A_2$ is the minimum angle of $T_1$, shown by the left plot in Figure \ref{fig:anisotrop_triangles}. 
If $\angle A_1A_3A_2 \rightarrow 0$, by the assumption, one of the two triangles $T_2$ and $T_3$ must be isotropic. Then, one of the paths $A_1$-$A_3$ and $A_1-A_2-A_3$ must have the desired property.  If $\angle A_1A_3A_2$ is also bounded below, and neither of $T_2$ and $T_3$ has the minimum angle $\theta_m$, we can estimate $\angle A_1A_3A_2$ by considering $T_4$. As $T_4$ has the minimum angle $\theta_m$ and has the size greater than $\rho h_{T_1}$, by sine law we know its edges have the minimum length $\sin(\theta_m) \rho h_{T_1} $. So $A_1A_2$ is bounded below by this quantity. As $\angle A_1A_3A_2$ is the minimum angle, $A_1A_2$ is also the edge with the minimum length. Therefore, using the sine law again, we have $\sin(\angle A_1A_3A_2) \ge \rho \sin(\theta_m) \sin(\theta_M)$, which implies that $T_1$ is isotropic with the minimum angle $ \arcsin (\rho \sin(\theta_m) \sin(\theta_M))$. So the path $A_1$-$A_3$ has the desired property.

Case 2. If $\angle A_1A_3A_2$ is not the minimum angle of $T_1$, without loss of generality, we suppose $\angle A_3A_1A_2$ is the minimum angle, shown by the right plot in Figure \ref{fig:anisotrop_triangles}. By the assumption, one of $T_2$ and $T_3$ must isotropic. Similarly, one of the paths $A_1-A_3$ and $A_1-A_2-A_3$ has the desired property.
\end{proof}


\section{Estimates regarding maximum angle conditions}

\begin{lemma}
\label{lem_grad}
Given a triangle $T$ with the maximum angle $\theta_M$, then there holds
\begin{equation}
\label{lem_grad_eq0}
\| \nabla v_h \|_{0, T} \le  h^{1/2}_T/\sqrt{2\sin(\theta_M)} \sum_{i=1,2,3} \| \nabla v_h \cdot \bft_i\|_{L^2(e_i)}, ~~~~ \forall v_h\in \mathcal{P}_1(T).
\end{equation}
\end{lemma} 
\begin{proof}
Let $R_T$ be the circumradius of $T$ and $\bft_i$ is a unit tangential vector of the edge $e_i$, $i=1,2,3$.
The cotangent formula \cite{2007Chen} and the law of sines gives
\begin{equation}
\| \nabla v_h \|_{0, T}^2 = R_T \sum_{i=1}^3  \cos(\theta_i) \| \nabla v_h \cdot \bft_i\|_{L^2(e_i)}^2, 
\end{equation}
which leads to the desired result by $R_T  \le h_T/(2\sin(\theta_M))$.
\end{proof}

\begin{lemma}
\label{lem_regular_edges}
Assume a tetrahedron $T$ has the maximum angle condition $\theta_M$. Then, $T$ has three edges (may not share one vertex) such that 
\begin{equation}
\label{lem_regular_edges_eq01}
|\emph{det}(M)|\ge c_m:=\min\{\sqrt{3}/2,\sin(\theta_M) \}\min\{ \cos(\theta_M/2),\sin(\theta_M) \}^2, 
\end{equation}
where $M$ is the matrix formed by the unit direct vectors of these edges. In addition, there holds 
\begin{equation}
\label{lem_regular_edges_eq02}
\frac{c_m}{6}|e_1||e_2||e_3| \le |T| \le \frac{1}{6}|e_1||e_2||e_3|.
\end{equation}
\end{lemma}
\begin{proof}
In $T=A_1A_2A_3A_4$, we first choose the edge $e_1$ such that the associated dihedral angle is the largest one, and without loss of generality, we assume $e_1=A_1A_2$, as shown in Figure \ref{fig:tetra_maxangle}. Let this dihedral angle be $\theta_1$, and let the directional vector of $e_1$ be $\bft_1$. By Lemma 6 in \cite{1992Michal}, we have $\theta_1\in[\pi/3,\theta_M]$. This edge has two neighbor elements $\triangle A_1A_2A_3$ and $\triangle A_1A_2A_4$. Then, we pick the edges $e_2$ and $e_3$ from these two faces such that they have the largest angle from $e_1$ in their faces, denoted by $\theta_2$ and $\theta_3$ respectively. Clearly, we have $\theta_2, \theta_3\in[(\pi-\theta_M)/2,\theta_M]$. Note that $e_2$ and $e_3$ may or may not share the same vertex, but the argument for both the two cases are the same. See Figure \ref{fig:tetra_maxangle} for illustration that they do not share a vertex. Let $\bft_2$ and $\bft_3$, respectively, be the directional vectors of $e_2$ and $e_3$. Let $M$ be the matrix formed by these three vectors. Let $\bfn$ be the norm vector to $\bft_1$ and $\bft_2$. 
Then, the direct calculation shows $\abs{ \bft_3\cdot\bfn} = \cos(\theta_1-\pi/2)\sin(\theta_3)$, and thus
\begin{equation}
\label{lem_regular_edges_eq1}
\abs{\text{det}(M)} = \abs{ (\bft_1 \times \bft_2)\cdot \bft_3 } = \sin(\theta_2) \abs{\bfn \cdot \bft_3} = \sin(\theta_1)\sin(\theta_2)\sin(\theta_3).
\end{equation}
Therefore, we obtain the estimates of $\abs{\text{det}(M)}$ by the upper and lower bounds of $\theta_i$, $i=1,2,3$. 

As for \eqref{lem_regular_edges_eq02}, without loss of generality, we consider the tetrahedron shown in Figure \ref{fig:tetra_maxangle}. Let $l$ be the distance from $A_4$ to the plane $\triangle A_1A_2A_3$. It is easy to see $l =|e_3|\sin(\theta_3)\sin(\theta_1)$. Then, we have
\begin{equation}
\label{lem_regular_edges_eq2}
|T| = |e_1||e_2|\sin(\theta_2)l/6 = |e_1||e_2||e_3| \sin(\theta_1)\sin(\theta_2)\sin(\theta_3)/6
\end{equation}
which yields \eqref{lem_regular_edges_eq02}.
\end{proof}

\begin{lemma}
\label{lem_edge_max}
Given a triangle $T$ with maximum angle $\theta_M$, let $e_1$ and $e_2$ be the two edges of $T$ adjacent to the maximum angle, then each for each segment $e \subseteq T$, there holds
\begin{equation}
\label{lem_edge_max_eq0}
\| \bfv_h\cdot\bft_e \|_{0,e} \lesssim  (\sin(\theta))^{-1/2}  \sum_{i=1,2} \| \bfv_h\cdot\bft_{e_i} \|_{0,e_i}, ~~~ \forall \bfv_h \in \left[ \mathcal{P}_0(T) \right]^2.
\end{equation}
\end{lemma}
\begin{proof}
We first consider a right-angle triangle. Clearly, $\theta_M = \pi/2$. Suppose $e_1$ is on the $x_1$ axis. Let the angle sandwiched by $e$ and $e_1$ be $\theta$. Then, we have $\bft_e = \cos(\theta) \bft_1 +  \sin(\theta) \bft_2$, and thus
\begin{equation}
\begin{split}
\label{lem_edge_max_eq1}
\| \bfv_h\cdot\bft_e \|_{0,e} & = |e|^{1/2} |\bfv_h\cdot\bft_e| \le |e|^{1/2} |\cos(\theta) | |\bfv_h\cdot\bft_1| + |e|^{1/2} |\sin(\theta) | |\bfv_h\cdot\bft_2| \\
& \le  \| \bfv_h\cdot\bft_{e_1} \|_{0,e_1} + \| \bfv_h\cdot\bft_{e_2} \|_{0,e_2},
\end{split}
\end{equation}
where we have used $|e|\cos(\theta) \le |e_1|$ and $|e|\sin(\theta) \le |e_2|$. See Figure \ref{fig:ebound} for illustration. Now, for a general triangle, we consider the affine mapping $\bfx = \mathfrak{F}(\hat{\bfx}) := A\hat{\bfx} = [\bft_1,\bft_2] \hat{\bfx}$. Clearly, there holds $\|A\|_2\lesssim 1$ and $\| A^{-1} \|_2\lesssim (\sin(\theta))^{-1}$.
Then, we obtain from \eqref{lem_edge_max_eq1} that
\begin{equation*}
\begin{split}
\label{lem_edge_max_eq2}
\| \bfv_h \cdot \bft_e \|_{0,e} & = |e|^{1/2} | \bfv_h \cdot \bft_e | \lesssim  |\hat{e}|^{1/2} | (A\bfv_h)\cdot \bft_{\hat{e}} | =  \| (A\bfv_h)\cdot \bft_{\hat{e}} \|_{0,\hat{e}}  \\
& \lesssim   \| (A\bfv_h)\cdot\bft_{\hat{e}_1} \|_{0,\hat{e}_1} + \| (A\bfv_h)\cdot\bft_{\hat{e}_2} \|_{0,\hat{e}_2}  
 \lesssim \| A^{-1} \|^{1/2} \left( |e_1|^{1/2} |\bfv_h\cdot \bft_{e_1}| + |e_2|^{1/2} |\bfv_h\cdot \bft_{e_2}| \right) \\
& \lesssim (\sin(\theta))^{-1/2} ( \| \bfv_h\cdot\bft_{e_1} \|_{0,e_1} + \| \bfv_h\cdot\bft_{e_2} \|_{0,e_2}),
\end{split}
\end{equation*}
which finishes the proof.
\end{proof}

\section{A Poincar\'e-type inequality on anisotropic elements}
\begin{lemma}
\label{lem_poin_anis}
Let $P$ be a convex polyhedron with $F$ being one of its faces, and let $l_F$ be the supporting height of $F$. Assume the projection of $P$ onto the plane containing $F$ is exactly $F$. Then, for $u\in H^1(P),$ there holds
\begin{equation}
\label{lem_poin_anis_eq0}
\| u  \|_{0,P} \lesssim l^{1/2}_F \| u \|^2_{0,F} + l_F \| \nabla u \|_{0,P}.
\end{equation}
\end{lemma}
\begin{proof}
Without loss of generality, we assume that $F$ is on the $x_1x_2$ plane. For each $\bfx = (\xi_1,\xi_2,\xi_3)\in P$, let $\bfx_F= (\xi_1,\xi_2,0)$ be the projection of $\bfx$ onto ${F}$, and let $l(\bfx_F)$ be the height at $\bfx_F$.  We can write $w(\bfx) - w(\bfx_F) = \int_{0}^{\xi_3} \partial_{x_3}w \dd x_3$ and derive
\begin{equation}
\begin{split}
\label{lem_poin_anis_eq1}
\| u \|^2_{0,P} &= \int_{{F}}  \int_0^{l(\bfx_F)}  \left( u(\bfx_F) + \int_{0}^{\xi_3} \partial_{x_3}u \dd x_3 \right)^2 \dd x_3 \dd \bfx_F \\
& \le 2 \int_{F} \int_0^{l(\bfx_F)}  | u(\bfx_F) |^2  \dd x_3 \dd \bfx_F +  \int_{F} \int_0^{l(\bfx_F)} \left( \int_{0}^{\xi_3} \partial_{x_3}u \dd x_3 \right)^2 \dd \xi_3  \dd \bfx_F   \\
& \le 2 l_F \| u \|^2_{0,F} + 2l^2_F \| \nabla u \|^2_{0,P}
\end{split}
\end{equation}
where in the last inequality we have also used H\"older's inequality.
\end{proof}

\end{appendices}

 


\bibliographystyle{abbrv}
\bibliography{vem.bib,RuchiBib.bib}

\end{document}